\documentclass[oneside,english]{article}
\usepackage[T1]{fontenc}
\usepackage[latin9]{inputenc}
\usepackage{float}
\usepackage{epic,eepic}
\usepackage{comment}
\usepackage{amstext,amsthm,amssymb,amsmath}
\usepackage{epsfig,subfigure,epstopdf}
\usepackage{graphicx}
\usepackage{subfigure}
\usepackage{wrapfig}
\usepackage{fullpage}
\usepackage{color}
\usepackage{multirow}
\usepackage{tabulary}
\usepackage{booktabs}
\usepackage{enumerate}
\usepackage{color}

\theoremstyle{definition}

\newtheorem{lemma}{Lemma}[section]

\makeatletter
\numberwithin{equation}{section}
\numberwithin{figure}{section}

\makeatother

\title{Contrast-independent partially explicit time discretizations for nonlinear multiscale problems}
\author{Eric T. Chung \footnote{Department of Mathematics, The Chinese University of Hong Kong (CUHK), Hong Kong SAR}, ~Yalchin Efendiev\footnote{Department of Mathematics, Texas A\&M University, College Station, TX 77843, USA \& North-Eastern Federal University, Yakutsk, Russia}, ~ Wing Tat Leung\footnote{Department of Mathematics, University of California, Irvine, USA}, ~ Wenyuan Li\footnote{Department of Mathematics, Texas A\&M University, College Station, TX 77843, USA}}

\usepackage{babel}
\begin{document}
\maketitle

\section*{Abstract}

This work continues a line of works on developing partially explicit methods
for multiscale problems. In our previous works, we have considered
linear multiscale problems, where  the spatial heterogeneities are
 at subgrid level and are not resolved. In these works, 
we have introduced contrast-independent partially explicit time 
discretizations for linear equations. 
The contrast-independent partially explicit time discretization 
 divides the spatial space into two components: contrast dependent (fast) and contrast independent (slow) spaces defined via multiscale space decomposition. 
Following this decomposition, temporal splitting is proposed that treats
fast components implicitly and slow components explicitly. The 
space decomposition and temporal splitting are chosen such that it guarantees
a stability and formulate a condition for the time stepping.
This condition is formulated as a condition on slow spaces. 
In this paper, we extend this approach to nonlinear problems.
We propose a splitting approach and
derive a condition that guarantees stability.
 This condition requires some type of contrast-independent 
spaces for slow components of the solution.
We present numerical results and show that the proposed methods provide results similar to implicit methods with the time step that is independent of the contrast.

\section{Introduction}

Nonlinear problems arise in many applications and they are typically described
by some nonlinear partial differential equations. In many applications,
these problems have multiscale nature and contain multiple scales
and high contrast. Examples include nonlinear porous media flows
(e.g., Richards' equations, Forchheimer flow, and so on, e.g.,
 \cite{ehlers2020darcy,bear2013dynamics}), 
where the media
properties contain many spatial scales and high contrast. 
Because of high contrast
in the media properties, these processes also occur on multiple times scales.
E.g.,
for nonlinear diffusion, the flow can be fast in
high conductivity regions, while the flow is slow 
in low conductivity regions. Because
of disparity of time scales, special temporal discretizations are
often sought, which is the main goal of the paper
in the context of multiscale problems.

When the media properties are high, the flow and transport become
fast and requires small time step to resolve the dynamics. Implicit 
discretization
can be used to handle fast dynamics; however, this requires
solving large-scale nonlinear systems. For nonlinear problems, explicit methods
are used when possible to avoid solving nonlinear systems.
 The main drawback of explicit methods is that they
require small time steps that scale as the fine mesh and depend
on physical parameters, e.g., the contrast. To alleviate this issue, 
we propose
a novel nonlinear splitting algorithm following our earlier works
\cite{chung_partial_expliict21,chung_partial_expliict_wave21}
 for
linear equations. The main idea of our approaches is to use multiscale
methods on a coarse spatial grid such that the time step scales as
the coarse mesh size.

Next, we give a brief overview of multiscale methods for spatial
discretizations that are used in our paper.
Multiscale spatial
algorithms have been extensively studied for linear and nonlinear problems.
For linear problems, many multiscale methods have been developed.
These include
homogenization-based approaches \cite{eh09,le2014msfem}, 
multiscale
finite element methods \cite{eh09,hw97,jennylt03}, 
generalized multiscale finite element methods (GMsFEM) \cite{chung2016adaptiveJCP,MixedGMsFEM,WaveGMsFEM,chung2018fast,GMsFEM13}, 
constraint energy minimizing GMsFEM (CEM-GMsFEM) 
\cite{chung2018constraint, chung2018constraintmixed}, nonlocal
multi-continua (NLMC) approaches \cite{NLMC},
metric-based upscaling \cite{oz06_1}, heterogeneous multiscale method 
\cite{ee03}, localized orthogonal decomposition (LOD) 
\cite{henning2012localized}, equation-free approaches \cite{rk07,skr06}, 
multiscale stochastic approaches \cite{hou2017exploring, hou2019model, hou2018adaptive},
and hierarchical multiscale method \cite{brown2013efficient}.
For high-contrast problems, 
approaches such as GMsFEM and NLMC are developed.
For example, in the  GMsFEM 
\cite{chung2018constraint}, multiple basis
functions or continua are designed to capture the multiscale features
due to high contrast
\cite{chung2018constraintmixed, NLMC}. 
These approaches require a careful design of multiscale
dominant modes. For nonlinear problems, linear multiscale basis functions
can be replaced by nonlinear maps \cite{ep03a,ep03c,efendiev2014generalized}.

Our proposed approaches follow concepts developed in 
\cite{chung_partial_expliict21,chung_partial_expliict_wave21} 
for linear equations. In these works, we design 
splitting algorithms for solving flow and wave equations. 
In both cases, the solution space is divided into two parts, coarse-grid
part, and the correction part. Coarse-grid solution is computed using
multiscale basis functions with CEM-GMsFEM. The correction part 
uses special spaces in the complement space (complement to the coarse space). 
A careful choice of these spaces guarantees that
the method is stable. Our analysis in \cite{chung_partial_expliict21,chung_partial_expliict_wave21}
shows that for the stability, the correction
space should be free of contrast and thus, this requires a special multiscale
space construction. These splitting algorithms take their origin from earlier
works \cite{marchuk1990splitting,VabishchevichAdditive}. 
In this paper, we extend the linear concepts to nonlinear problems.

Splitting algorithms for nonlinear problems have often been used in 
the literature.
Earlier approaches include implicit-explicit approaches and other techniques
\cite{ascher1997implicit,li2008effectiveness,abdulle2012explicit,engquist2005heterogeneous,ariel2009multiscale,narayanamurthi2019epirk,shi2019local,duchemin2014explicit,frank1997stability,izzo2017highly,ruuth1995implicit,hundsdorfer2007imex,du2019third}. In many approaches, nonlinear contributions are roughly
divided into
two parts depending on whether it is
 easy to  implicitly solve discretized system.
For easy to solve part, implicit discretization is used, while
for the rest, explicit discretization is used.
However, in general, one can not separate these parts for the problems
under the consideration.  
 Our goal is to use
splitting concepts and treat implicitly and explicitly some
parts of the solution. As a result, we can use larger time steps
that scale as the coarse mesh size.

Our approach starts a nonlinear dynamical system
\[
u_t + f(u) + g(u)=0,
\]
where $f(u)$ represents diffusion-like operator, while $g(u)$ represents 
reaction-like terms.
In linear problems, for the stability, we formulate a condition
that involves the time step, the energy and $L^2$ norm of the solution
in complement space. This is a constraint for the time step. With an
appropriate choice of the complement space, this condition guarantees
the stability for the time steps that scale as the coarse mesh size.
To obtain similar conditions for nonlinear problems, we carry out
the analysis for nonlinear $f(u)$ and $g(u)$ functions. 
The analysis reveals conditions that are required for stability.
The conditions of multiscale spaces turn out to share some
 similarities to those
for nonlinear multiscale methods \cite{efendiev2014generalized}.
We remark several observations. 

\begin{itemize}

\item Additional degrees of freedom  are needed for dynamic problems,
in general, to handle
missing information.

\item We note that restrictive time step scales as the coarse mesh
size and, thus, much coarser.

\end{itemize}

We present several  numerical results. 
In our numerical results, we consider two cases.
In the first case, we take the reaction term, $g(u)$,
to be nonlinear, while the diffusion term, $f(u)$,
to be linear. In the second case, we consider
both to be nonlinear and use for $f(u)$ the form
$f(u)=-div(\kappa(x,u)\nabla u)$.
The media properties for the diffusion is taken to be heterogeneous
and we choose smooth and singular source terms.
We compare the proposed approach to the approach, where all
degrees of freedom are handled implicitly.
We show that the proposed methods provide an approximation 
similar to fully implicit methods.

The paper is organized as follows. In the next section, we provide
some preliminaries. In Section 3, we present our main assumptions
and stability estimates for fine-grid problem. In Section 4, we
describe the partially explicit method and its stability, following
some discussions in Section 5.
Numerical results are presented in Section 6.

\section{Problem Setting}

We consider the following equation 
\begin{align}
u_{t}=-f(u)-g(u), \label{eq1}
\end{align}
where $f=\cfrac{\delta F}{\delta u}$ and $g=\cfrac{\delta G}{\delta u}$ which
are the variational derivative of energies $F(u):=\int_{\Omega}E_{1}(u)$
and $G(u):=\int_{\Omega}E_{2}(u)$.
Here, $f$ is assumed to be contrast dependent nonlinear (or linear) (i.e.,
$f$ introduces stiffness in the system)
while $g$ is contrast independent (i.e., $g$ does not introduce stiffness).

We assume $f(u)\in V^{*}$ and $g(u)\in L^{2}(\Omega)$ for all $u\in V$.
We can then consider the weak formulation of the problem, namely,
finding $u\in V$ such that
\[
(u_{t},v)=-(f(u),v)_{V^{*},V}\;\forall v\in V.
\]
To simplify the notation, we will simply write $(\cdot,\cdot)$ instead
of $(\cdot,\cdot)_{V^{*},V}$ in the following discussion.

Example 1: For $F(u)=\cfrac{1}{2}\int_{\Omega}\kappa|\nabla u|^{2}$
and $G(u)=0$, we will have $f(u')=\cfrac{\delta F}{\delta u}\in H^{1}(\Omega)=(H^{1}(\Omega))^{*}$
for $V=H^{1}(\Omega)$.
\[
(\cfrac{\delta F}{\delta u}(u'),v)=\int_{\Omega}\kappa\nabla u'\cdot\nabla v,
\]
 and thus, we have 
\[
-f(u')=\nabla\cdot(\kappa\nabla u').
\]
In strong form, we have
\[
u_{t}=\nabla\cdot(\kappa\nabla u).
\]

Example 2: For $F=\cfrac{1}{p}\int\kappa|\nabla u|^{p}$ and $G(u)=0$,
we will have $f(u')=\cfrac{\delta F}{\delta u}\in W^{1,\frac{p}{p-1}}(\Omega)=(W^{1,p}(\Omega))^{*}$
for $V=W^{1,p}(\Omega)$
\[
(\cfrac{\delta F}{\delta u}(u'),v)=\int\kappa|\nabla u'|^{p-2}\nabla u'\cdot\nabla v,
\]
and thus, we have 
\[
-f(u')=\nabla\cdot(\kappa|\nabla u'|^{p-2}\nabla u').
\]
In strong form, we have
\[
u_{t}=\nabla\cdot(\kappa|\nabla u|^{p-2}\nabla u).
\]

In the following discussion, we will make the following assumptions
on the second variational derivative of $F$ and $G$.
\begin{itemize}
\item The second variational derivative $\delta^{2}F$
and $\delta^{2}G$ satisfy
\[
\delta^{2}F(u)(v,v)\geq c(u)\|v\|_{V}^{2}\;\forall u,v\in V
\]
\[
\delta^{2}G(u)(v,v)\geq b(u)\|v\|^{2}\;\forall u,v\in V,
\]
where $0\leq c(u)<\infty$ and $-\underline{b}\leq b(u)<\infty$ are
independent on $v$.
\item The second variational derivative $\delta^{2}F$
and $\delta^{2}G$ are bounded. That is,
\[
|\delta^{2}F(u)(w,v)|\leq C(u)\|v\|_{V}\|w\|_{V}\;\forall u,v,w\in V
\]
\[
|\delta^{2}G(u)(w,v)|\leq B\|v\|_{L_{2}}\|w\|_{L_{2}}\;\forall u,v,w\in V,
\]
where $0<C(u)<\infty$ and $0<B<\infty$ are independent on $v,w$.
\end{itemize}

\section{Discretization}

To solve the problem, a standard method is finite element
approach. We can consider the numerical solution $u_{H}(t,\cdot)\in V_{H}$
satisfies 
\begin{align}
(u_{H,t},v)=-(f(u)+g(u),v)\;\forall v\in V_{H}, \label{eq:CEM_problem}
\end{align}
where $V_{H}$ is a finite element space in $V$. 

For the time discretization, we can consider two simplest discretizations
which are forward Euler and backward Euler methods.
For the forward Euler method, we consider $\{u_{H}^{k}\}_{k=0}^{N}\subset V_{H}$
such that 
\[
(\cfrac{u_{H}^{n+1}-u_{H}^{n}}{\Delta t},v)+(f(u_{H}^{n})+g(u_{H}^{n}),v)=0\;\forall v\in V_{H}.
\]
For the backward Euler method, we consider $\{u_{H}^{k}\}_{k=0}^{N}\subset V_{H}$
such that 
\[
(\cfrac{u_{H}^{n+1}-u_{H}^{n}}{\Delta t},v)+(f(u_{H}^{n+1})+g(u_{H}^{n+1}),v)=0\;\forall v\in V_{H}.
\]
Next, we would like to derive  stability conditions for backward and forward
 Euler methods.

Since
\[
F(u_{H}^{n})=F(u_{H}^{n+1})-(f(u_{H}^{n+1}),u_{H}^{n+1}-u_{H}^{n})+\cfrac{1}{2}\delta^{2}F(\xi_{1}^{n})(u_{H}^{n+1}-u_{H}^{n},u_{H}^{n+1}-u_{H}^{n})
\]
and 
\[
G(u_{H}^{n})=G(u_{H}^{n+1})-(g(u_{H}^{n+1}),u_{H}^{n+1}-u_{H}^{n})+\cfrac{1}{2}\delta^{2}G(\xi_{2}^{n})(u_{H}^{n+1}-u_{H}^{n},u_{H}^{n+1}-u_{H}^{n})
\]
for some $\xi_{i}^{n}=(1-\lambda_{i})u_{H}^{n+1}+\lambda_{i}u_{H}^{n}$
with $\lambda_{i}\in(0,1)$ and $i=1,2$, we have
\begin{align*}
0= & (\cfrac{u_{H}^{n+1}-u_{H}^{n}}{\Delta t},u_{H}^{n+1}-u_{H}^{n})+(f(u_{H}^{n+1})+g(u_{H}^{n+1}),u_{H}^{n+1}-u_{H}^{n})\\
= & \cfrac{1}{\Delta t}\|u_{H}^{n+1}-u_{H}^{n}\|^{2}+F(u_{H}^{n+1})-F(u_{H}^{n})+G(u_{H}^{n+1})-G(u_{H}^{n})\\
 & +\cfrac{1}{2}\delta^{2}F(\xi_{1}^{n})(u_{H}^{n+1}-u_{H}^{n},u_{H}^{n+1}-u_{H}^{n})+\cfrac{1}{2}\delta^{2}G(\xi_{2}^{n})(u_{H}^{n+1}-u_{H}^{n},u_{H}^{n+1}-u_{H}^{n})\\
\geq & \Big(\cfrac{1}{\Delta t}+b(u)\Big)\|u_{H}^{n+1}-u_{H}^{n}\|^{2}+F(u_{H}^{n+1})-F(u_{H}^{n})+G(u_{H}^{n+1})-G(u_{H}^{n})+\cfrac{c(u)}{2}\|u_{H}^{n+1}-u_{H}^{n}\|_{V}^{2}.
\end{align*}
We have 
\begin{align*}
F(u_{H}^{n+1})+G(u_{H}^{n+1}) & \leq F(u_{H}^{n+1})+G(u_{H}^{n+1})+\cfrac{c(u)}{2}\|u_{H}^{n+1}-u_{H}^{n}\|_{V}^{2}+\Big(\cfrac{1}{\Delta t}-\underline{b}\Big)\|u_{H}^{n+1}-u_{H}^{n}\|^{2}\\
 & \leq F(u_{H}^{n})+G(u_{H}^{n})
\end{align*}
 for any $\Delta t$ and thus, backward Euler method is stable if
$\Delta t\underline{b}\leq1$.

Similarly, for forward Euler method, we can use
\[
F(u_{H}^{n+1})=F(u_{H}^{n})+(f(u_{H}^{n}),u_{H}^{n+1}-u_{H}^{n})+\cfrac{1}{2}\delta^{2}F(\xi_{1}^{n})(u_{H}^{n+1}-u_{H}^{n},u_{H}^{n+1}-u_{H}^{n})
\]
\[
G(u_{H}^{n+1})=G(u_{H}^{n})+(g(u_{H}^{n}),u_{H}^{n+1}-u_{H}^{n})+\cfrac{1}{2}\delta^{2}G(\xi_{2}^{n})(u_{H}^{n+1}-u_{H}^{n},u_{H}^{n+1}-u_{H}^{n})
\]
and obtain 
\begin{align*}
0= & (\cfrac{u_{H}^{n+1}-u_{H}^{n}}{\Delta t},u_{H}^{n+1}-u_{H}^{n})+(f(u^{n})+g(u_{H}^{n}),u_{H}^{n+1}-u_{H}^{n})\\
= & \cfrac{1}{\Delta t}\|u_{H}^{n+1}-u_{H}^{n}\|^{2}+F(u_{H}^{n+1})-F(u_{H}^{n})+G(u_{H}^{n+1})-G(u_{H}^{n})\\
 & -\cfrac{1}{2}\delta^{2}F(\xi_{1}^{n})(u_{H}^{n+1}-u_{H}^{n},u_{H}^{n+1}-u_{H}^{n})-\cfrac{1}{2}\delta^{2}G(\xi_{2}^{n})(u_{H}^{n+1}-u_{H}^{n},u_{H}^{n+1}-u_{H}^{n})\\
\geq &(\cfrac{1}{\Delta t}-B)\|u_{H}^{n+1}-u_{H}^{n}\|^{2}+F(u_{H}^{n+1})-F(u_{H}^{n})+ G(u_{H}^{n+1})-G(u_{H}^{n})-\cfrac{C(\xi^{n})}{2}\|u_{H}^{n+1}-u_{H}^{n}\|_{V}^{2}.
\end{align*}
Therefore, if $\Delta t\Big(\cfrac{C(\xi)}{2}\cfrac{\|u_{H}^{n+1}-u_{H}^{n}\|_{V}^{2}}{\|u_{H}^{n+1}-u_{H}^{n}\|^{2}}+B\Big)\leq1$
for any $\xi=(1-\lambda)u_{H}^{k+1}+\lambda u_{H}^{k}$ with $0\leq k\leq N-1$,
we have
\begin{align*}
F(u_{H}^{n+1})+G(u_{H}^{n+1}) & \leq F(u_{H}^{n})+G(u_{H}^{n}).
\end{align*}

We can see that although forward Euler method is easier for implementation,
we require a small time step for stability if $\sup_{v\in V_{H}}\cfrac{\|v\|_{V}^{2}}{\|v\|^{2}}$
or $C(\xi)$ is large.

We remark that in typical cases we will consider  $\underline{b}$
and $B$ are not too large. Therefore, we have $\Delta t\underline{b}$
and $\Delta tB$ is small and  the energy $G$ will not affect
the stability too much.

\section{Partially explicit scheme with space splitting}

To obtain an efficient method, one can consider partially explicit
scheme by splitting finite element space. We consider $V_{H}$ is
a direct sum of two subspace $V_{H,1}$ and $V_{H,2}$, namely, $V_{H}=V_{H,1}\oplus V_{H,2}$.
The finite element solution is then satisfying 
\begin{align*}
(u_{H,1,t}+u_{H,2,t},v_{1})+(f(u_{H,1}+u_{H,2})+g(u_{H,1}+u_{H,2}),v_{1}) & =0\;\forall v_{1}\in V_{H,1},\\
(u_{H,1,t}+u_{H,2,t},v_{2})+(f(u_{H,1}+u_{H,2})+g(u_{H,1}+u_{H,2}),v_{2}) & =0\;\forall v_{2}\in V_{H,2},
\end{align*}
where $u_{H}=u_{H,1}+u_{H,2}$. We can use a partially explicit time
discretization. For example, we can consider 
\begin{align*}
(\cfrac{u_{H,1}^{n+1}-u_{H,1}^{n}}{\Delta t}+\cfrac{u_{H,2}^{n}-u_{H,2}^{n-1}}{\Delta t},v_{1})+(f(u_{H,1}^{n+1}+u_{H,2}^{n})+g(u_{H,1}^{n}+u_{H,2}^{n}),v_{1}) & =0\;\forall v_{1}\in V_{H,1},\\
(\cfrac{u_{H,1}^{n}-u_{H,1}^{n-1}}{\Delta t}+\cfrac{u_{H,2}^{n+1}-u_{H,2}^{n}}{\Delta t},v_{2})+(f(u_{H,1}^{n+1}+u_{H,2}^{n})+g(u_{H,1}^{n}+u_{H,2}^{n}),v_{2}) & =0\;\forall v_{2}\in V_{H,2}.
\end{align*}

\subsection*{Energy stability}
\begin{lemma}
\label{lem:lem1}
If 
\begin{equation}
\begin{split}
(f(u_{H}^{n+1})-f(u_{H,1}^{n+1}+u_{H,2}^{n}),u_{H}^{n+1}-u_{H}^{n})\leq\\
\cfrac{\bar{c}}{2}\|u_{H}^{n+1}-u_{H}^{n}\|_{V}^{2}+\Big(\cfrac{(1-\gamma)}{\Delta t}-(1+\gamma)\cfrac{B}{2}\Big)\sum_{i}\|u_{H,i}^{n+1}-u_{H,i}^{n}\|^{2},
\end{split}
\end{equation}
where $\bar{c}=\inf_{u\in V_{H}}c(u)$ and $\gamma=\sup_{v_{1}\in V_{H,1},v_{2}\in V_{H,2}}\cfrac{(v_{1},v_{2})}{\|v_{1}\|\|v_{2}\|}<1$,
we have 
\[
\cfrac{\gamma}{2\Delta t}\sum_{i}\|u_{H,i}^{n+1}-u_{H,i}^{n}\|^{2}+F(u_{H}^{n+1})+G(u_{H}^{n+1})\leq\cfrac{\gamma}{2\Delta t}\sum_{i}\|u_{H,i}^{n}-u_{H,i}^{n-1}\|^{2}+F(u_{H}^{n})+G(u_{H}^{n}).
\]
\end{lemma}

\begin{proof}
By substitute $v_{1}=u_{H,1}^{n+1}-u_{H,1}^{n}$ and $v_{2}=u_{H,2}^{n+1}-u_{H,2}^{n}$,
we have 
\begin{align*}
\cfrac{1}{\Delta t}\|u_{H,1}^{n+1}-u_{H,1}^{n}\|^{2}+\cfrac{1}{\Delta t}(u_{H,2}^{n}-u_{H,2}^{n-1},u_{H,1}^{n+1}-u_{H,1}^{n})\\
+(f(u_{H,1}^{n+1}+u_{H,2}^{n})+g(u_{H}^{n}),u_{H,1}^{n+1}-u_{H,1}^{n}) & =0,
\end{align*}
and 
\begin{align*}
\cfrac{1}{\Delta t}\|u_{H,2}^{n+1}-u_{H,2}^{n}\|^{2}+\cfrac{1}{\Delta t}(u_{H,1}^{n}-u_{H,1}^{n-1},u_{H,2}^{n+1}-u_{H,2}^{n})\\
+(f(u_{H,1}^{n+1}+u_{H,2}^{n})+g(u_{H}^{n}),u_{H,2}^{n+1}-u_{H,2}^{n}) & =0.
\end{align*}
Summing up the above two equations, we have 
\begin{align*}
\cfrac{1}{\Delta t}\sum_{i}\|u_{H,i}^{n+1}-u_{H,i}^{n}\|^{2}+\cfrac{1}{\Delta t}\sum_{i\neq j}(u_{H,i}^{n}-u_{H,i}^{n-1},u_{H,j}^{n+1}-u_{H,j}^{n})\\
+(f(u_{H,1}^{n+1}+u_{H,2}^{n})+g(u_{H}^{n}),u_{H}^{n+1}-u_{H}^{n}) & =0.
\end{align*}
We first use 
\begin{align*}
\cfrac{1}{\Delta t}|\sum_{i\neq j}(u_{H,i}^{n}-u_{H,i}^{n-1},u_{H,j}^{n+1}-u_{H,j}^{n})| & \leq\cfrac{\gamma}{\Delta t}\sum_{i\neq j}\|u_{H,i}^{n}-u_{H,i}^{n-1}\|\|u_{H,j}^{n+1}-u_{H,j}^{n}\|\\
 & \leq\cfrac{\gamma}{2\Delta t}\sum_{i}\Big(\|u_{H,i}^{n+1}-u_{H,i}^{n}\|^{2}+\|u_{H,i}^{n}-u_{H,i}^{n-1}\|^{2}\Big)
\end{align*}
and obtain 
\begin{align*}
 & \cfrac{1}{\Delta t}\sum_{i}\|u_{H,i}^{n+1}-u_{H,i}^{n}\|^{2}+\cfrac{1}{\Delta t}\sum_{i\neq j}(u_{H,i}^{n}-u_{H,i}^{n-1},u_{H,j}^{n+1}-u_{H,j}^{n})\\
\geq & \cfrac{2-\gamma}{2\Delta t}\sum_{i}\|u_{H,i}^{n+1}-u_{H,i}^{n}\|^{2}-\cfrac{\gamma}{2\Delta t}\sum_{i}\|u_{H,i}^{n}-u_{H,i}^{n-1}\|^{2}.
\end{align*}

To prove the stability of the method, we can consider 
\[
F(u_{H}^{n})=F(u_{H}^{n+1})-(f(u_{H}^{n+1}),u_{H}^{n+1}-u_{H}^{n})+\cfrac{1}{2}\delta^{2}F(\xi_{1}^{n})(u_{H}^{n+1}-u_{H}^{n},u_{H}^{n+1}-u_{H}^{n})
\]
\[
G(u_{H}^{n+1})=G(u_{H}^{n})+(g(u_{H}^{n}),u_{H}^{n+1}-u_{H}^{n})+\cfrac{1}{2}\delta^{2}G(\xi_{2}^{n})(u_{H}^{n+1}-u_{H}^{n},u_{H}^{n+1}-u_{H}^{n})
\]
for some $\xi_{i}^{n}=(1-\lambda_{i})u_{H}^{n+1}+\lambda_{i}u_{H}^{n}$
with $\lambda_{i}\in(0,1)$ and $i=1,2$.

Therefore, we have 
\begin{equation}
\begin{split}
 &(f(u_{H,1}^{n+1}+u_{H,2}^{n}),u_{H}^{n+1}-u_{H}^{n})\\
=  &(f(u_{H,1}^{n+1}+u_{H,2}^{n})-f(u_{H}^{n+1}),u_{H}^{n+1}-u_{H}^{n})\\
 &+F(u_{H}^{n+1})-F(u_{H}^{n})+
\cfrac{1}{2}\delta^{2}F(\xi_{1}^{n})(u_{H}^{n+1}-u_{H}^{n},u_{H}^{n+1}-u_{H}^{n})
\end{split}
\end{equation}
and 
\begin{align*}
 & (g(u_{H}^{n}),u_{H}^{n+1}-u_{H}^{n})\\
= & G(u_{H}^{n+1})-G(u_{H}^{n})-\cfrac{1}{2}\delta^{2}G(\xi_{2}^{n})(u_{H}^{n+1}-u_{H}^{n},u_{H}^{n+1}-u_{H}^{n}).
\end{align*}
Thus, we obtain 
\begin{align*}
 & \cfrac{\gamma}{2\Delta t}\sum_{i}\|u_{H,i}^{n+1}-u_{H,i}^{n}\|^{2}+\cfrac{(1-\gamma)}{\Delta t}\sum_{i}\|u_{H,i}^{n+1}-u_{H,i}^{n}\|^{2}+F(u_{H}^{n+1})+G(u_{H}^{n+1})+\cfrac{c(\xi^{n})}{2}\|u_{H}^{n+1}-u_{H}^{n}\|_{V}^{2}\\
\leq & \cfrac{\gamma}{2\Delta t}\sum_{i}\|u_{H,i}^{n}-u_{H,i}^{n-1}\|^{2}+F(u_{H}^{n})+G(u_{H}^{n})+\cfrac{B}{2}\|u_{H}^{n+1}-u_{H}^{n}\|^{2}\\
 & +(f(u_{H}^{n+1})-f(u_{H,1}^{n+1}+u_{H,2}^{n}),u_{H}^{n+1}-u_{H}^{n})
\end{align*}
and 
\[
\cfrac{B}{2}\|u_{H}^{n+1}-u_{H}^{n}\|^{2}\leq(1+\gamma)\cfrac{B}{2}\sum_{i} \|u_{H,i}^{n+1}-u_{H,i}^{n}\|^{2}.
\]

If 
\[
(f(u_{H}^{n+1})-f(u_{H,1}^{n+1}+u_{H,2}^{n}),u_{H}^{n+1}-u_{H}^{n})\leq\cfrac{c(\xi^{n})}{2}\|u_{H}^{n+1}-u_{H}^{n}\|_{V}^{2}+\Big(\cfrac{(1-\gamma)}{\Delta t}-(1+\gamma)\cfrac{B}{2}\Big)\sum_{i}\|u_{H,i}^{n+1}-u_{H,i}^{n}\|^{2}
\]
then we have 
\[
\cfrac{\gamma}{2\Delta t}\sum_{i}\|u_{H,i}^{n+1}-u_{H,i}^{n}\|^{2}+F(u_{H}^{n+1})+G(u_{H}^{n+1})\leq\cfrac{\gamma}{2\Delta t}\sum_{i}\|u_{H,i}^{n}-u_{H,i}^{n-1}\|^{2}+F(u_{H}^{n})+G(u_{H}^{n}).
\]
\end{proof}
\begin{lemma}
\label{lem:lem2}
If 
\begin{align}
\cfrac{\bar{C}_{2}^{2}}{2\bar{c}}\sup_{v_{2}\in V_{H,2}}\cfrac{\|v_{2}\|_{V}^{2}}{\|v_{2}\|^{2}}+(1+\gamma)\cfrac{B}{2}\leq\cfrac{(1-\gamma)}{\Delta t}, 
\label{eq:stab_cond}
\end{align}
where $\bar{c}=\inf_{u\in V_{H}}c(u)$, $\bar{C}_{2}=\sup_{\xi \in V_{H}}C_{2}(\xi)$
and 
\[
C_{2}(\xi)=\sup_{v\in V_{H},w\in V_{H,2}}\cfrac{1}{\|v\|_{V}\|w\|_{V}}\delta^{2}F(\xi)(w,v)\leq C(\xi),
\]
 we have 
\[
\cfrac{\gamma}{2\Delta t}\sum_{i}\|u_{H,i}^{n+1}-u_{H,i}^{n}\|^{2}+F(u_{H}^{n+1})+G(u_{H}^{n+1})\leq\cfrac{\gamma}{2\Delta t}\sum_{i}\|u_{H,i}^{n}-u_{H,i}^{n-1}\|^{2}+F(u_{H}^{n})+G(u_{H}^{n}).
\]
\end{lemma}

\begin{proof}
For the proof of this lemma, we will show that if the condition of 
Lemma \ref{lem:lem2} holds, then the condition of Lemma \ref{lem:lem1} holds.
For this reason, we will need to estimate 
$(f(u_{H}^{n+1})-(f(u_{H,1}^{n+1}+u_{H,2}^{n}),u_{H}^{n+1}-u_{H}^{n})$.
Similar to the proof in previous lemma, we consider 
\[
(f(u_{H}^{n+1}),u_{H}^{n+1}-u_{H}^{n})=(f(u_{H,1}^{n+1}+u_{H,2}^{n}),u_{H}^{n+1}-u_{H}^{n})+(\delta^{2}F(\tilde{\xi}^{n})(u_{H}^{n+1}-u_{H}^{n}),u_{H,2}^{n+1}-u_{H,2}^{n})
\]
for some $\tilde{\xi}^{n}=(1-\tilde{\lambda})u_{H}^{n+1}+\tilde{\lambda}(u_{H,1}^{n+1}+u_{H,2}^{n})$
with $\tilde{\lambda}\in(0,1)$. We have 
\begin{align*}
(f(u_{H}^{n+1})-f(u_{H,1}^{n+1}+u_{H,2}^{n}),u_{H}^{n+1}-u_{H}^{n}) & =(\delta^{2}F(\tilde{\xi}^{n})(u_{H}^{n+1}-u_{H}^{n}),u_{H,2}^{n+1}-u_{H,2}^{n})\\
 & \leq C_{2}(\tilde{\xi}^{n})\|u_{H}^{n+1}-u_{H}^{n}\|_{V}\|u_{H,2}^{n+1}-u_{H,2}^{n}\|_{V}.
\end{align*}
Since 
\[
C_{2}(\tilde{\xi}^{n})\|u_{H}^{n+1}-u_{H}^{n}\|_{V}\|u_{H,2}^{n+1}-u_{H,2}^{n}\|_{V}\leq\cfrac{\bar{c}}{2}\|u_{H}^{n+1}-u_{H}^{n}\|_{V}^{2}+\cfrac{C_{2}^{2}(\tilde{\xi}^{n})}{2\bar{c}}\|u_{H,2}^{n+1}-u_{H,2}^{n}\|_{V}^{2},
\]
 we have 
\[
(f(u_{H}^{n+1})-f(u_{H,1}^{n+1}+u_{H,2}^{n}),u_{H}^{n+1}-u_{H}^{n})\leq\cfrac{\bar{c}}{2}\|u_{H}^{n+1}-u_{H}^{n}\|_{V}^{2}+\cfrac{C_{2}^{2}(\tilde{\xi}^{n})}{2\bar{c}}\|u_{H,2}^{n+1}-u_{H,2}^{n}\|_{V}^{2}.
\]
If 
\[
\cfrac{\bar{C}_{2}^{2}}{2\bar{c}}\cfrac{\|u_{H,2}^{n+1}-u_{H,2}^{n}\|_{V}^{2}}{\|u_{H,2}^{n+1}-u_{H,2}^{n}\|^{2}}+(1+\gamma)\cfrac{B}{2}\leq\cfrac{(1-\gamma)}{\Delta t},
\]
we have the condition formulated in Lemma \ref{lem:lem1}
\begin{equation}
\begin{split}
(f(u_{H}^{n+1})-f(u_{H,1}^{n+1}+u_{H,2}^{n}),u_{H}^{n+1}-u_{H}^{n})\leq\\
\cfrac{\bar{c}}{2}\|u_{H}^{n+1}-u_{H}^{n}\|_{V}^{2}+\Big(\cfrac{(1-\gamma)}{\Delta t}-(1+\gamma)\cfrac{B}{2}\Big)\sum_{i}\|u_{H,i}^{n+1}-u_{H,i}^{n}\|^{2}.
\end{split}
\end{equation}
By Lemma \ref{lem:lem1}, we get the result.
\end{proof}
Example 1. For $F=\cfrac{1}{2}\int_{\Omega}\kappa|\nabla u|^{2}$
and $G(u)=0$, we have
\[
(\cfrac{\delta F}{\delta u},v)=\int_{\Omega}\kappa\nabla u\cdot\nabla v,
\]
and 
\[
\delta^{2}F(u)(w,v)=\int_{\Omega}\kappa\nabla v\cdot\nabla w\;\forall u\in V,
\]
 and thus, we have 
\[
\bar{c}=C_{2}^{2}=1,\;B=0
\]
and the partially explicit scheme is stable when 
\[
\cfrac{\Delta t}{2}\sup_{v_{2}\in V_{H,2}}\cfrac{\|\kappa^{\frac{1}{2}}\nabla v_{2}\|^{2}}{\|v_{2}\|^{2}}\leq(1-\gamma)\;\forall v_{2}\in V_{H,2}.
\]

\section{Discussions}

\subsection{$G=0$ case}

First, we present some discussions for $G=0$ case. In this case, the first 
stability condition
for partial explicit scheme is
\begin{equation}
\begin{split}
(f(u_{H,1}^{n+1}+u_{H,2}^{n+1})-f(u_{H,1}^{n+1}+u_{H,2}^{n}),u_{H}^{n+1}-u_{H}^{n})\leq\\
\cfrac{\bar{c}}{2}\|u_{H}^{n+1}-u_{H}^{n}\|_{V}^{2}+
\cfrac{(1-\gamma)}{\Delta t}\sum_{i}\|u_{H,i}^{n+1}-u_{H,i}^{n}\|^{2}.
\end{split}
\end{equation}
This condition can be understood as a nonlinear 
constraint on the ``second space'' that 
represents $u_{H,2}^n$ and in order to have a small bound, 
one needs to guarantee that
$u_{H,1}^n$ captures important degrees of freeom. Indeed, the smallness
of
\[
{(f(u_{H,1}^{n+1}+u_{H,2}^{n+1})-f(u_{H,1}^{n+1}+u_{H,2}^{n}),u_{H}^{n+1}-u_{H}^{n})\over 
\|u_{H}^{n+1}-u_{H}^{n}\|_{V}^{2}}
\]
is a condition on $u_{H,1}^n$ (on the coarse space)
and requires that this term is chosen such that
the difference is independent of the contrast. 
This condition is more evident in Lemma \ref{lem:lem2}, where
the condition on $V_2$ is
\[
\cfrac{\bar{C}_{2}^{2}}{2\bar{c}}\sup_{v_{2}\in V_{H,2}}\cfrac{\|v_{2}\|_{V}^{2}}{\|v_{2}\|^{2}}\leq\cfrac{(1-\gamma)}{\Delta t}.
\]

\subsection{$G\not = 0 $ case.}
In this case, we will first treat the nonlinear forcing explicitly and we will also discuss the case when $g(u)$ is partially explicit.

\section{Numerical Results}
\label{sec:num}

In this section, we will present numerical results for various cases.
We will consider several choices for $f(u)$ and $g(u)$. For $f(u)$, we will
use diffusion operator for 
linear case
\[
 f(u) = - \nabla \cdot (\kappa \nabla u),
\]
and nonlinear case 
\begin{equation}
\label{eq:numres}
f(u) = - \nabla \cdot (\kappa \alpha(u) \nabla u). 
\end{equation}
In all examples, we will use two heterogeneous high contrast $\kappa(x)$
that represent the media, where one is more complex (more channels).
As for $g(u)$, we will consider several choices of nonlinear reaction terms
as discussed below. This term will contain nonlinear reaction and 
steady state spatial source term. One source term will be more regular
and the other more singular. The singular source term is chosen so
that CEM solution requires additional basis functions as the source term
contains subgrid features.
In all numerical examples, 
the coarse mesh size is $\frac{1}{10}$ and the fine 
mesh size is $\frac{1}{100}$. 
For the time discretization, we will consider the final time $T=0.05$.
In our numerical tests, we will compare three methods.
\begin{itemize}

\item First, we will use implicit CEM to compute the 
solution without additional degrees of freedom (called ``Implicit CEM'' 
in our graphs).

\item Secondly, we will compute the solution with additional degrees
of freedom using implicit CEM (called ``Implicit CEM with additional
basis'' in our graphs).

\item Finally, we will compute the solution with additional degrees of fredoom
using our proposed partially explicit approach (called ``Partially Explicit Splitting CEM'' in our graphs).

\end{itemize} 

In all examples, we use Newton or Picard iterations to 
find the solution of nonlinear equations.
In all examples, we observe that our proposed partially explicit
method provides similar accuracy as the implicit CEM approach that uses additional
degrees of freedom.

\subsection{$V_{H,1}$ and $V_{H,2}$ constructions}

In this section, we present a way to construct
the spaces satisfying (\ref{eq:stab_cond}) based on linear
problems. These spaces are constructed under the assumption
that linear multiscale structure can be used to accurately model
coarse-grid solution. It can be shown that the linear spaces
satisfy (\ref{eq:stab_cond}) under some assumptions on $\kappa(x,u)$
(see (\ref{eq:numres})).
Here, we follow our previous work \cite{chung_partial_expliict21}.
As the constrained energy minimization basis functions are 
constructed such that
they are almost orthogonal to a space $\tilde{V}$, 
the CEM finite element space is a good option for $V_{H,1}$.
To find a $V_{H,2}$ satisfying the condition (\ref{eq:stab_cond}),
we can use an eigenvalue problem to construct
the local basis functions.
We will first introduce the CEM finite element space,
followed by the discussion of constructing $V_{H,2}$.
In the following, we let $V(S) = H_0^1(S)$ for a proper 
subset $S\subset \Omega$.

\subsubsection{CEM method}
\label{sec:cem}

In this section, we introduce the CEM method for solving the problem
(\ref{eq:CEM_problem}). 
We will construct the finite element space by solving
a constrained energy minimization problem. Let $\mathcal{T}_{H}$
be a coarse grid partition of $\Omega$. For $K_{i}\in\mathcal{T}_{H}$,
we first need to define a set of auxiliary basis functions in $V(K_{i})$.
We solve 
\begin{align*} \int_{K_i} \kappa \nabla \psi_j^{(i)} \cdot \nabla v = \lambda_j^{(i)} s_i ( \psi_j^{(i)},v) \;
\forall v \in V(K_i), \end{align*}
where \begin{align*} s_i(u,v) = \int_{K_i} \tilde{\kappa} u v, \;
\tilde{\kappa} = \kappa H^{-2} \; \text{or} \; 
\tilde{\kappa} = \kappa \sum_{i}\left|\nabla \chi_{i}\right|^{2}   \end{align*}  
with $\{\chi_i\}$ being a partition of unity functions corresponding to an overlapping partition of the domain.
We then collect the first $L_i$ eigenfunctions corresponding to the first $L_i$ smallest eigenvalues. 
We define \[ V_{aux}^{(i)}:=\text{span}\{\psi_{j}^{(i)}:\;1\leq j\leq L_{i}\}.  \]
We next define a  projection operator $\Pi:L^{2}(\Omega)\mapsto V_{aux}\subset L^{2}(\Omega)$
\[
s(\Pi u,v)=s(u,v)\;\forall v\in V_{aux}:=\sum_{i=1}^{N_{e}}V_{aux}^{(i)},
\]
where $s(u,v):=\sum_{i=1}^{N_{e}}s_{i}(u|_{K_{i}},v|_{K_{i}})$ and $N_e$ is the number of coarse elements.
We let $K_{i}^{+}$ be an oversampling domain of $K_{i}$, which is a few coarse blocks larger than $K_i$ \cite{chung2018constraint}. 
For each auxiliary basis functions $\psi_{j}^{(i)}$, we can find
a local basis function $\phi_{j}^{(i)}\in V(K_{i}^{+})$
such that 
\begin{align*}
a(\phi_{j}^{(i)},v)+s(\mu_{j}^{(i)},v) & =0\;\forall v\in V(K_{i}^{+}),\\
s(\phi_{j}^{(i)},\nu) & =s(\psi_{j}^{(i)},\nu)\;\forall\nu\in V_{aux}(K_{i}^{+})
\end{align*}
for some $\mu_{j}^{(i)} \in V_{aux}$.
We then define the space $V_{cem}$ as 
\begin{align*}
V_{cem} & :=\text{span}\{\phi_{j}^{(i)}:\;1\leq i\leq N_{e},1\leq j\leq L_{i}\}.
\end{align*}
The CEM solution $u_{cem}$ is given by
\begin{align*}
(\cfrac{ \partial u_{cem} }{\partial t} ,v ) + ( f(u_{cem}) + g(u_{cem}) , v ) = 0 \; \forall v\in V_{cem}.
\end{align*}
Let $\tilde{V}:= \{v\in V: \; \Pi(v) = 0 \}$ and we can now construct $V_{H,2}$.

\subsubsection{Construction of $V_{H,2}$}





The construction of $V_{H,2}$ is based on the CEM type finite
element space. For
each coarse element $K_{i}$, we will solve an eigenvalue problem to get the second type of auxiliary
basis. We obtain eigenpairs $(\xi_{j}^{(i)},\gamma_{j}^{(i)})\in(V(K_{i})\cap\tilde{V})\times\mathbb{R}$ by solving
\begin{align}
\label{eq:spectralCEM2}
\int_{K_{i}}\kappa\nabla\xi_{j}^{(i)}\cdot\nabla v & =\gamma_{j}^{(i)}\int_{K_{i}}\xi_{j}^{(i)}v, \;\ \forall v\in V(K_{i})\cap\tilde{V}
\end{align}
and rearranging the eigenvalues by $\gamma_1^{(i)}\leq \gamma_2^{(i)}\leq \cdots $.
For each $K_i$, we choose the first few $J_i$ eigenfunctions corresponding to the smallest $J_i$ eigenvalues. We define $V_{aux,2} := \text{span}\{\xi_j^{(i)} : 1\leq i \leq N_e, 1\leq j\leq J_i \}$.
For each auxiliary basis function $\xi_j^{(i)} \in V_{aux,2}$, we define a basis function $\zeta_{j}^{(i)} \in V(K_i^+)$ such
that $\mu_{j}^{(i),1} \in V_{aux,1}$, $ \mu_{j}^{(i),2} \in V_{aux,2}$ and 
\begin{align}
a(\zeta_{j}^{(i)},v)+s(\mu_{j}^{(i),1},v)+ ( \mu_{j}^{(i),2},v) & =0, \;\forall v\in V(K_i^+), \label{eq:v2a} \\
s(\zeta_{j}^{(i)},\nu) & =0, \;\forall\nu\in V_{aux,1}, \label{eq:v2b} \\
(\zeta_{j}^{(i)},\nu) & =( \xi_{j}^{(i)},\nu), \;\forall\nu\in V_{aux,2}, \label{eq:v2c}
\end{align}
where we use the notation $V_{aux,1}$ to denote the space $V_{aux}$ defined in Section \ref{sec:cem}.
We define $$V_{H,2}=\text{span}\{\zeta_{j}^{(i)}| \; 1\leq i \leq N_e, \;  1 \leq j\leq J_i\}.$$

\subsection{Linear $f(u)$}

In this subsection, we discuss the numerical results for  
\[ f(u) = - \nabla \cdot (\kappa \nabla u). \]
Equation $(\ref{eq1})$ becomes
\begin{align}
u_t - \nabla \cdot (\kappa \nabla u) + g(u) = 0. \label{eq71}
\end{align}  
For the time discretization, we consider the time step $\Delta t = \frac{T}{500} = 10^{-4}$. 

Let $u_h$ be the fine mesh solution for Equation $(\ref{eq71})$. We use Newton's method to solve the following implicit equation. 
\[ (\frac{u^{n+1}_h-u^n_h}{\Delta t},v) +  a(u^{n+1}_h,v)
+ (g(u^{n+1}_h),v)=0 \quad \forall v \in V_h, \]
where $a(u^{n+1}_h,v)=\int_{\Omega} \kappa \nabla u^{n+1}_h \cdot \nabla v $ and $(\cdot,\cdot)$ is the $L^2$ inner product.
In finite element methods, let $\{\varphi_i\}_i$ be fine mesh basis functions.
Let $m$ be the step number in Newton's method.
We have $u^{n+1,m+1}_h = \sum\limits_i U^{n+1,m+1}_{h,i} \varphi_i $, $u^{n+1,m}_h = \sum\limits_i U^{n+1,m}_{h,i} \varphi_i $ and $u^n_h = \sum\limits_i U^n_{h,i} \varphi_i$. Le $M$ and $A$ be the mass and stiffness matrices, respectively. 
Let $U^{n+1,m+1}_h = (U^{n+1,m+1}_{h,i})$, $U^{n+1,m}_h = (U^{n+1,m}_{h,i})$ and $U^n_h = (U^n_{h,i})$. 
We define
\begin{align*}
P(U^{n+1,m}_h) = M U^{n+1,m}_h + \Delta t \cdot A U^{n+1,m}_h + \Delta t \cdot \mathcal{G}
- M U^n_h,
\end{align*}
where $\mathcal{G}=(\mathcal{G}_i)$
\[
\mathcal{G}_i=(g(u^{n+1,m}_h),\varphi_i). 
\]
Then \begin{align*}
(JP)(U^{n+1,m}_h) = M + \Delta t\cdot A + \Delta t \cdot (J\mathcal{G}),
\end{align*} 
where $J\mathcal{G} = ((J\mathcal{G})_{ij})$
\[ (J\mathcal{G})_{ij} =  \frac{\partial (g(u^{n+1,m}_h),\varphi_i)}{\partial U^{n+1,m}_{h,j}}. \]
Then we have \[ U^{n+1,m+1}_h = U^{n+1,m}_h - (JP)^{-1}(U^{n+1,m}_h) P(U^{n+1,m}_h).   \]
Newton's method for coarse mesh is similar.
The partially explicit scheme is:
\begin{align*}
(\frac{u^{n+1}_{H,1}-u^n_{H,1}}{\Delta t} + \frac{u^n_{H,2}-u^{n-1}_{H,2}}{\Delta t} ,v_1)
+ a((u^{n+1}_{H,1}+u^n_{H,2}),v_1 )
+ ( g(u^n_{H,1} + u^n_{H,2}) , v_1 ) =0 \quad \forall v_1 \in V_{H,1}, \\
( \frac{u^{n+1}_{H,2}-u^n_{H,2}}{\Delta t} + \frac{u^n_{H,1}-u^{n-1}_{H,1}}{\Delta t} , v_2 ) + a((u^{n+1}_{H,1} +u^n_{H,2})  , v_2 ) + ( g(u^n_{H,1} + u^n_{H,2}) , v_2 ) =0 \quad \forall v_2 \in V_{H,2}.
\end{align*} 

In our first example, we consider \[g(u) = -(10\cdot u \cdot (u^2 - 1 ) + g_0). \]
In Figure~\ref{NRfig1}, the permeability field $\kappa$ and 
$g_0$ are presented. 
As is shown, this permeability field has
heterogeneous high contrast channels 
and $g_0$ is a singular source term. 
In Figure~\ref{NRfig2}, we first present the reference solution  
which is implicitly solved using fine grid basis functions. 
The middle plot in Figure~\ref{NRfig2} is implicit CEM solution
obtained with additional basis functions and 
the solution in the right plot is obtained using the partially 
explicit scheme presented above. 
These three plots all show the solution at $t=T$. 
We present two relative error plots in Figure~\ref{NRfig3}.
The first one is the relative $L^2$ error plot and 
the second one is the relative energy error plot.
The blue, red and black curves (in both plots) stand for the relative 
error for implicit CEM solution, implicit CEM solution (with
additional basis) and partially explicit solution
respectively.
In each of these two plots, there is a noticeable improvement for error when we use additional basis functions.
We find that 
the black curve coincides with the red curve,
which means that the partially explicit scheme 
can achieve similar accuracy as the fully implicit scheme.
\begin{figure}[H]
\centering
\subfigure{
\includegraphics[width = 6cm]{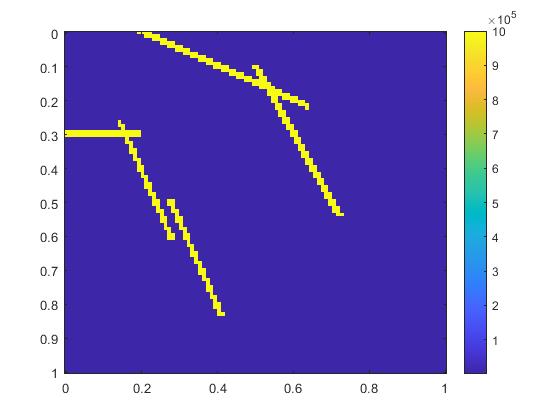}
}
\subfigure{
\includegraphics[width = 6cm]{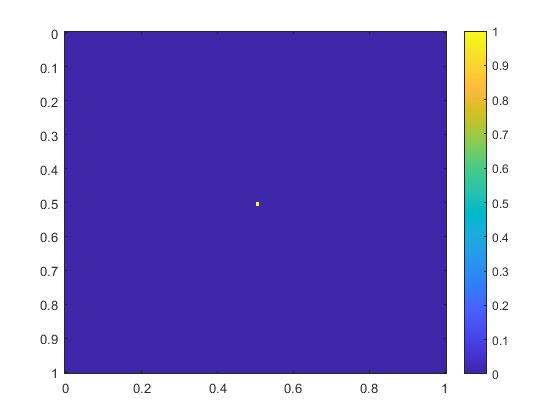}
}
\caption{Left: $\kappa$. Right: $g_0$.}
\label{NRfig1}
\end{figure}

\begin{figure}[H]
\centering
\subfigure{
\includegraphics[width = 5cm]{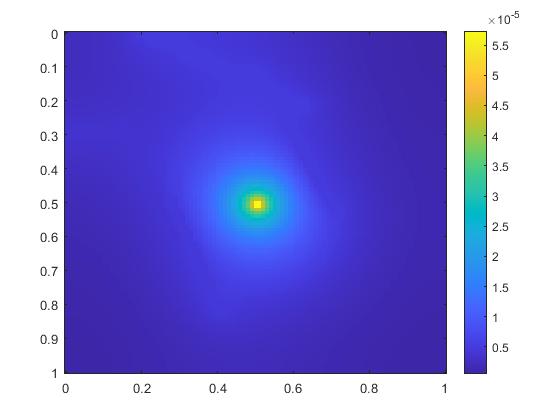}
}
\subfigure{
\includegraphics[width = 5cm]{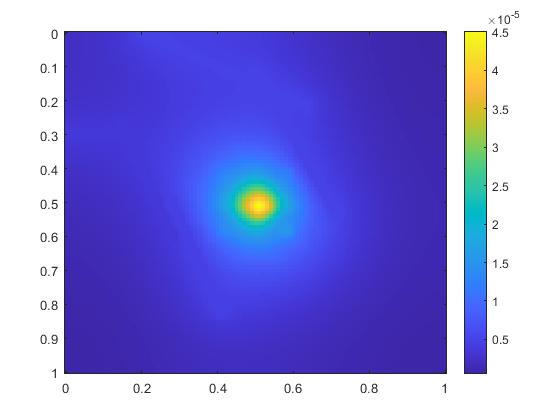}
}
\subfigure{
\includegraphics[width = 5cm]{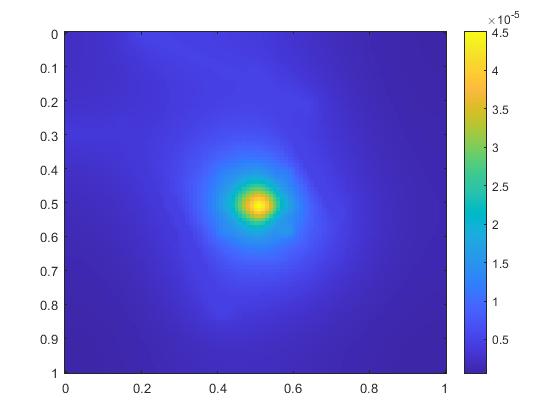}
}
\caption{Left: Reference solution at $t=T$. 
Middle: Implicit CEM solution (with additional basis) at $t=T$.
Right: Partially explicit solution at $t=T$.}
\label{NRfig2}
\end{figure}

\begin{figure}[H]
\centering
\subfigure{
\includegraphics[width = 6cm]{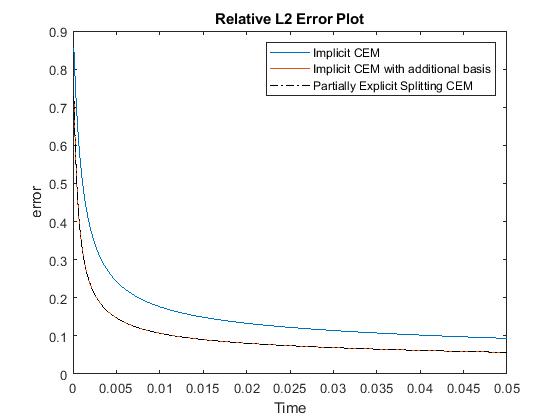}
}
\subfigure{
\includegraphics[width = 6cm]{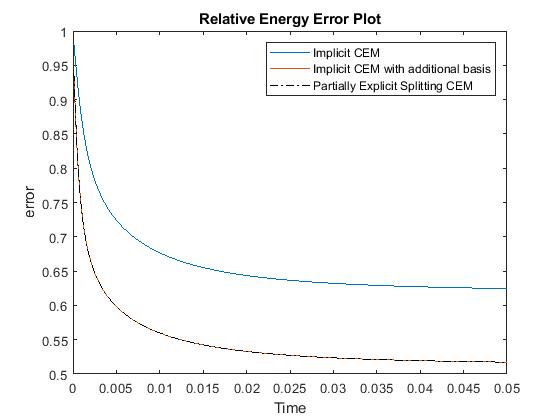}
}
\caption{Left: Relative $L^2$ error. 
Right: Relative energy error.}
\label{NRfig3}
\end{figure}

In this case, \[g(u) = -(10\cdot u \cdot (u^2 - 1 ) + g_0).\]
The difference is that we use a smooth source term.
In Figure~\ref{NRfig4}, the permeability field $\kappa$ and 
source term $g_0$ are shown.
The reference solution at the final time, implicit CEM solution (with additional 
basis) at the final time and partially explicit solution
at the final time 
are presented in Figure~\ref{NRfig5}. 
We show the relative $L^2$ error plot and 
the relative energy error plot in Figure~\ref{NRfig6}.
We see that the relative $L^2$ and energy error curves for 
implicit CEM (with additional basis) and
partially explicit scheme almost coincide,
which implies similar accuracy between them.
\begin{figure}[H]
\centering
\subfigure{
\includegraphics[width = 6cm]{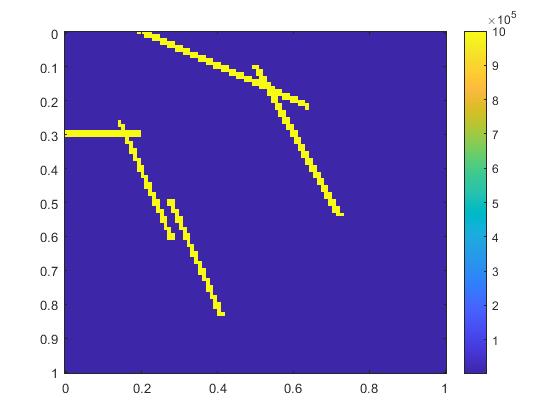}
}
\subfigure{
\includegraphics[width = 6cm]{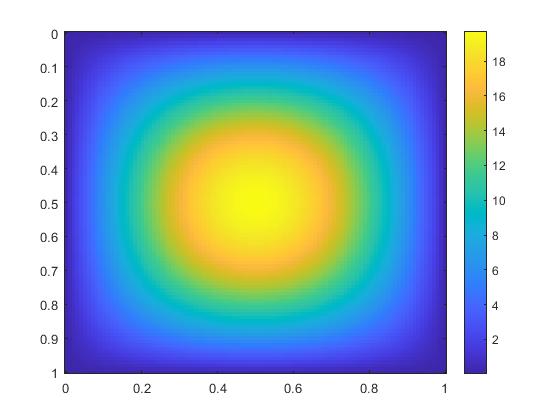}
}
\caption{Left: $\kappa$. Right: $g_0$.}
\label{NRfig4}
\end{figure}

\begin{figure}[H]
\centering
\subfigure{
\includegraphics[width = 5cm]{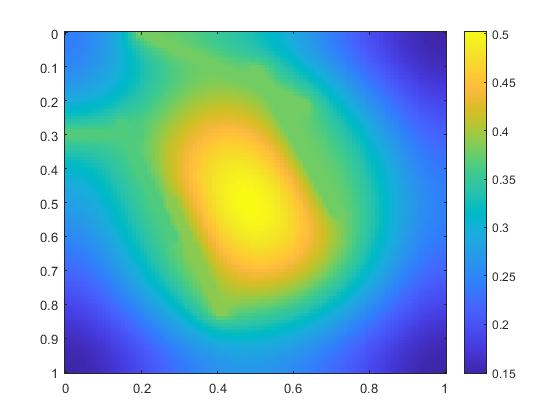}
}
\subfigure{
\includegraphics[width = 5cm]{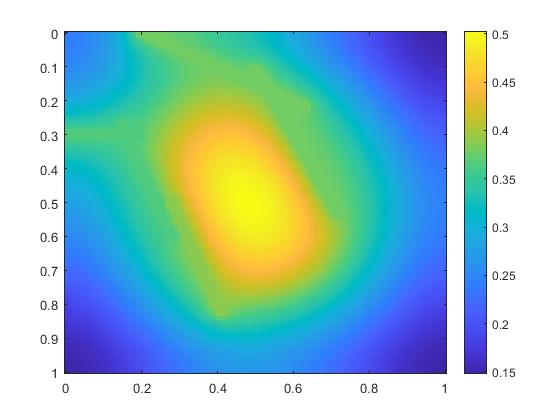}
}
\subfigure{
\includegraphics[width = 5cm]{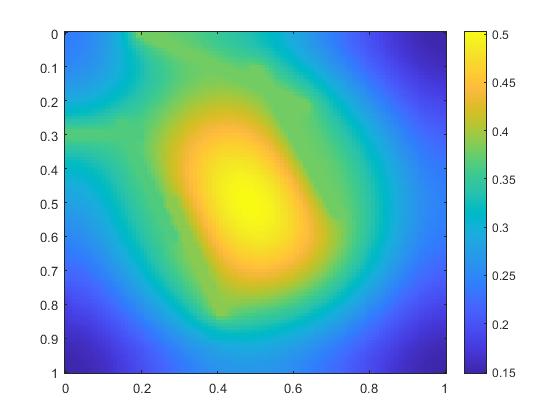}
}
\caption{Left: Reference solution at $t=T$. 
Middle: Implicit CEM solution (with additional basis) at $t=T$.
Right: Partially explicit solution at $t=T$.}
\label{NRfig5}
\end{figure}

\begin{figure}[H]
\centering
\subfigure{
\includegraphics[width = 6cm]{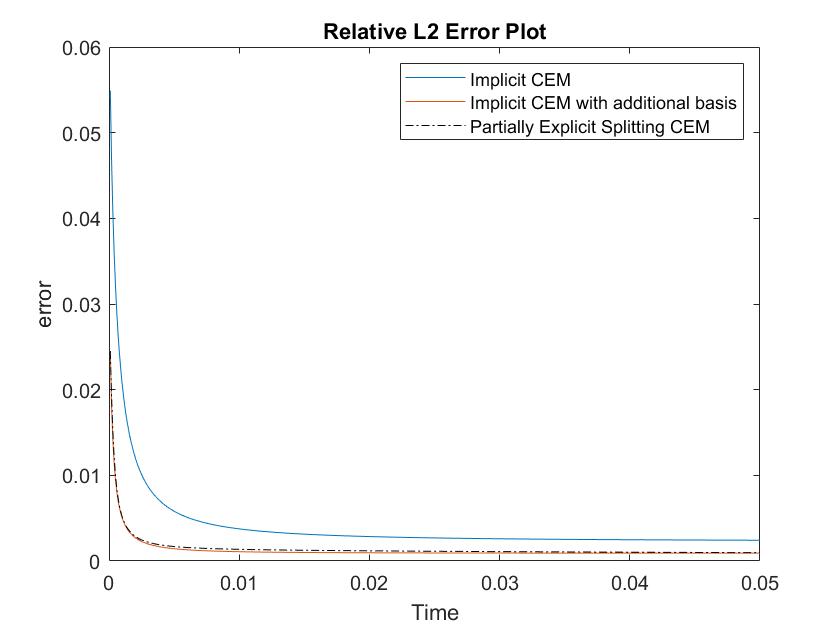}
}
\subfigure{
\includegraphics[width = 6cm]{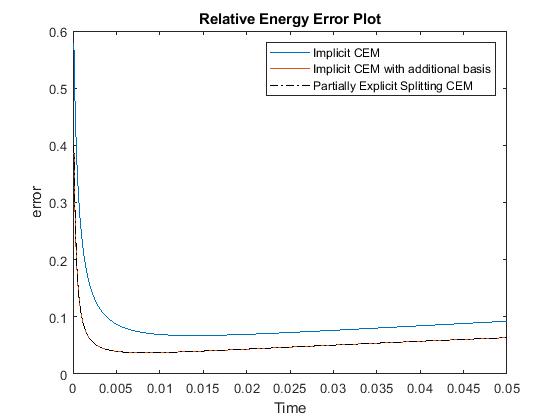}
}
\caption{Left: Relative $L^2$ error. 
Right: Relative energy error.}
\label{NRfig6}
\end{figure}

In our third test, \[g(u) = -(10\cdot u \cdot (u^2 - 1 ) + g_0),\]
and we use a more complicated 
permeability field with more high contrast channels.
Figure~\ref{NRfig7} shows the permeability field $\kappa$ 
and source term $g_0$. 
The reference solution at $t=T$, implicit CEM solution (with additional
basis) at $t=T$ and partially explicit solution at $t=T$ are presented in Figure~\ref{NRfig8}.
In Figure~\ref{NRfig9}, we show the relative $L^2$ error plot and the
relative energy error plot.
In this case, in both error plots, the curves for 
implicit CEM (with additional basis) and partially explicit
scheme coincide.
\begin{figure}[H]
\centering
\subfigure{
\includegraphics[width = 6cm]{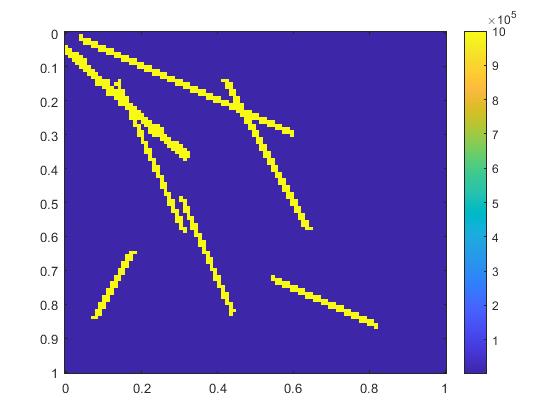}
}
\subfigure{
\includegraphics[width = 6cm]{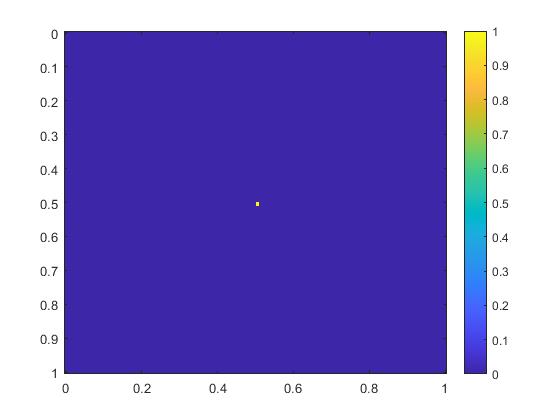}
}
\caption{Left: $\kappa$. Right: $g_0$.}
\label{NRfig7}
\end{figure}

\begin{figure}[H]
\centering
\subfigure{
\includegraphics[width = 5cm]{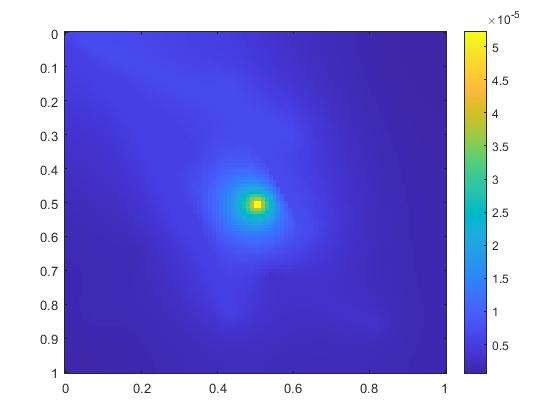}
}
\subfigure{
\includegraphics[width = 5cm]{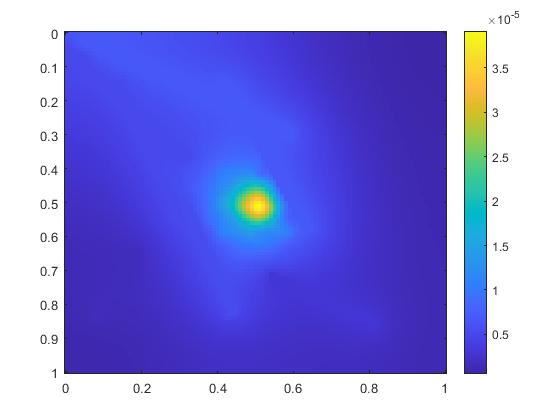}
}
\subfigure{
\includegraphics[width = 5cm]{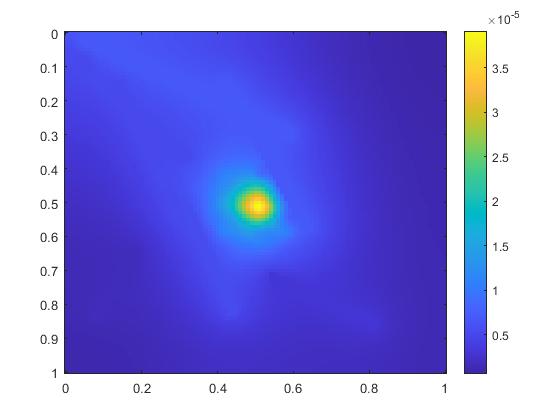}
}
\caption{Left: Reference solution at $t=T$. 
Middle: Implicit CEM solution (with additional basis) at $t=T$.
Right: Partially explicit solution at $t=T$.}
\label{NRfig8}
\end{figure}

\begin{figure}[H]
\centering
\subfigure{
\includegraphics[width = 6cm]{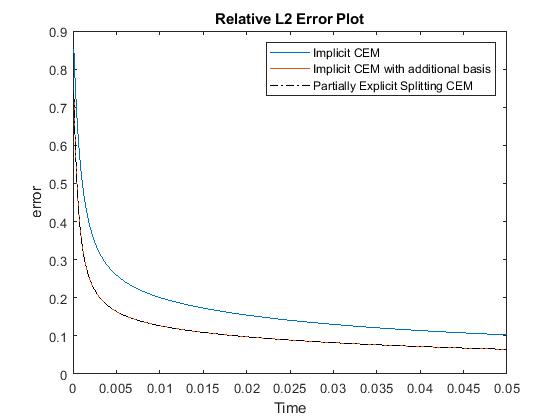}
}
\subfigure{
\includegraphics[width = 6cm]{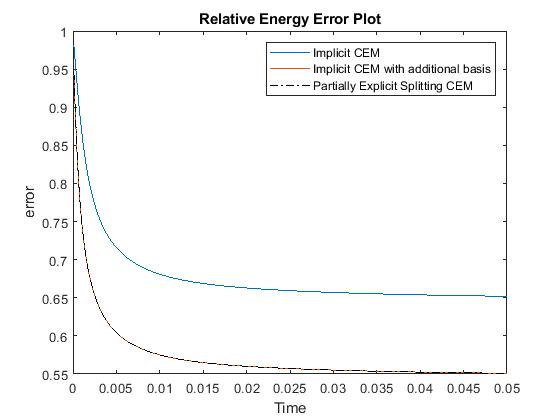}
}
\caption{Left: Relative $L^2$ error. 
Right: Relative energy error.}
\label{NRfig9}
\end{figure}

In this example, \[g(u) = -(10\cdot u \cdot (u^2 - 1 ) + g_0).\]
We use the more complicated permeability field 
and the smooth source term which are shown in Figure~\ref{NRfig10}.
In Figure~\ref{NRfig11}, the reference solution at the final time, 
implicit CEM solution (with additional basis) at the final time and 
partially explicit solution at the final time are presented.
Relative $L^2$ error and energy error plots 
are shown in Figure~\ref{NRfig12}.
In this case, the relative error for implicit CEM scheme is small 
and comparable to the two schemes with additional basis.
From Figure~\ref{NRfig12}, we can find that the $L^2$ and energy
error for implicit CEM (with additional basis) and
partially explicit scheme are nearly the same.
\begin{figure}[H]
\centering
\subfigure{
\includegraphics[width = 6cm]{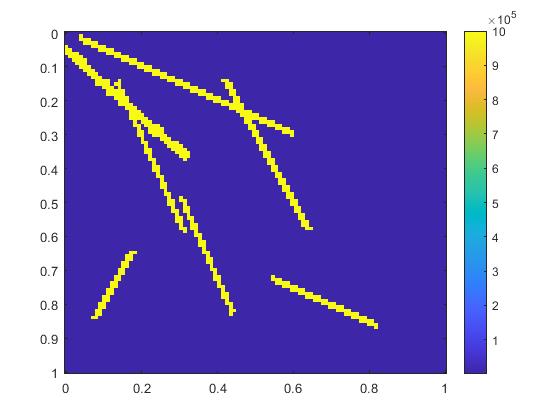}
}
\subfigure{
\includegraphics[width = 6cm]{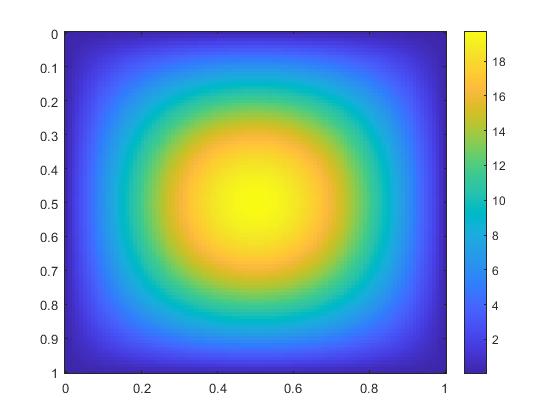}
}
\caption{Left: $\kappa$. Right: $g_0$.}
\label{NRfig10}
\end{figure}

\begin{figure}[H]
\centering
\subfigure{
\includegraphics[width = 5cm]{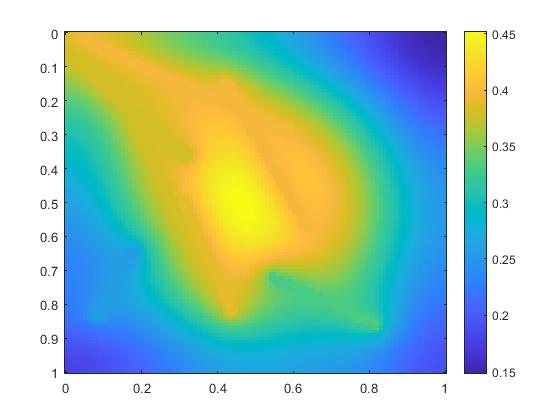}
}
\subfigure{
\includegraphics[width = 5cm]{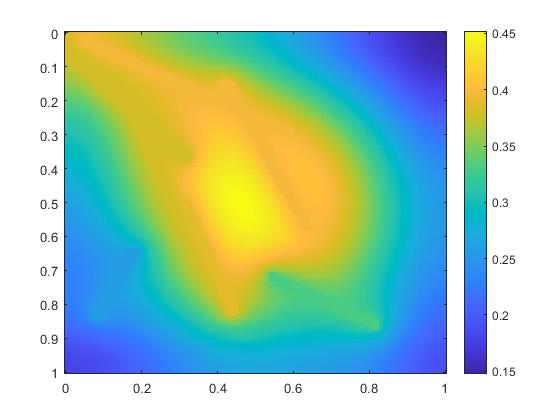}
}
\subfigure{
\includegraphics[width = 5cm]{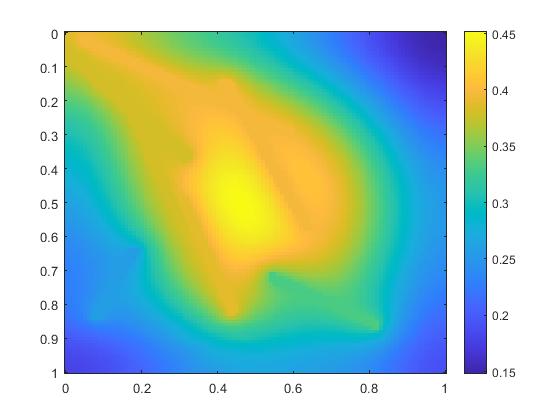}
}
\caption{Left: Reference solution at $t=T$. 
Middle: Implicit CEM solution (with additional basis) at $t=T$.
Right: Partially explicit solution at $t=T$.}
\label{NRfig11}
\end{figure}

\begin{figure}[H]
\centering
\subfigure{
\includegraphics[width = 6cm]{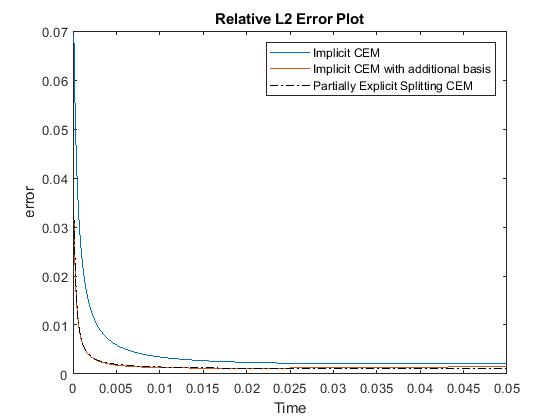}
}
\subfigure{
\includegraphics[width = 6cm]{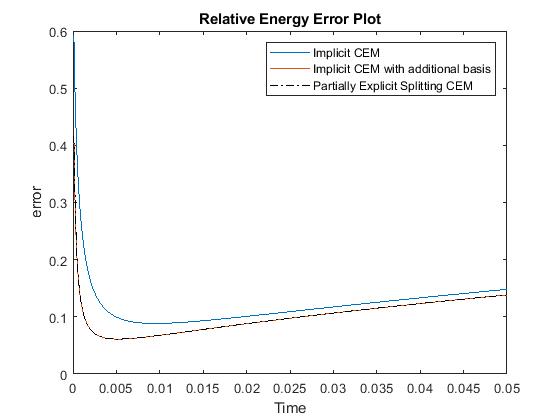}
}
\caption{Left: Relative $L^2$ error. 
Right: Relative energy error.}
\label{NRfig12}
\end{figure}

In this case, we use a new reaction term \[g = -(1+\cos(a_1 \cdot u) + g_0),\; a_1(x,y) = 2\cos(20\pi x)\cos(20\pi y).\]
In numerical experiments, we set $a_1$ to be constant inside every fine element.
Figure~\ref{NRfig13} shows the permeability field $\kappa$,
the source term $g_0$ and the function $a_1$.
The reference solution, implicit CEM solution (with additional basis) and partially explicit 
solution at the final time are shown in Figure~\ref{NRfig14}.
The relative $L^2$ error plot 
and relative energy error plot are presented in Figure~\ref{NRfig15}.
From the relative $L^2$ error plot, we find that the relative $L^2$ error for three schemes 
are comparable.
For the relative energy error, we can find a large improvement when we introduce 
additional basis.
The relative energy error for implicit CEM solution (with additional basis) 
and partially explicit solution are nearly the same.
\begin{figure}[H]
\centering
\subfigure{
\includegraphics[width = 5cm]{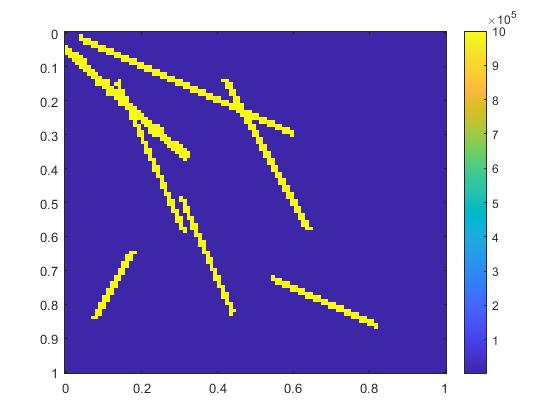}
}
\subfigure{
\includegraphics[width = 5cm]{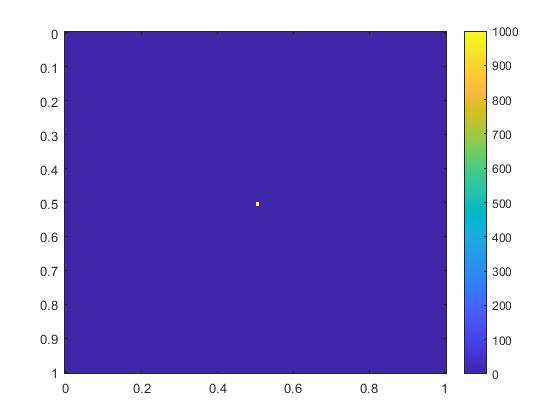}
}
\subfigure{
\includegraphics[width = 5cm]{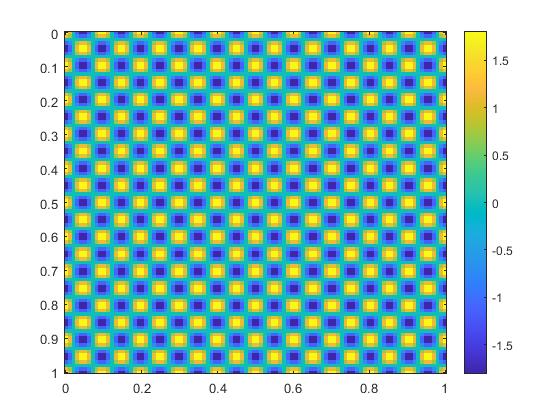}
}
\caption{Left: $\kappa$. Middle: $g_0$. Right: $a_1$.}
\label{NRfig13}
\end{figure}

\begin{figure}[H]
\centering
\subfigure{
\includegraphics[width = 5cm]{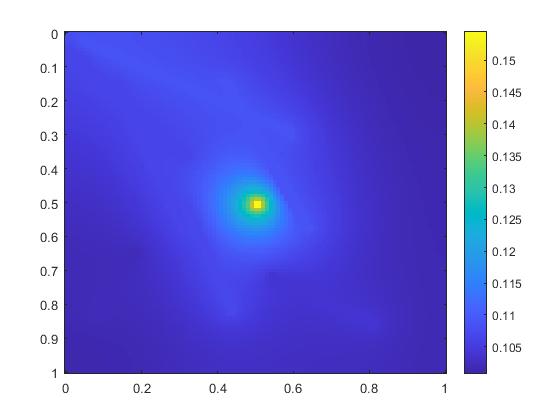}
}
\subfigure{
\includegraphics[width = 5cm]{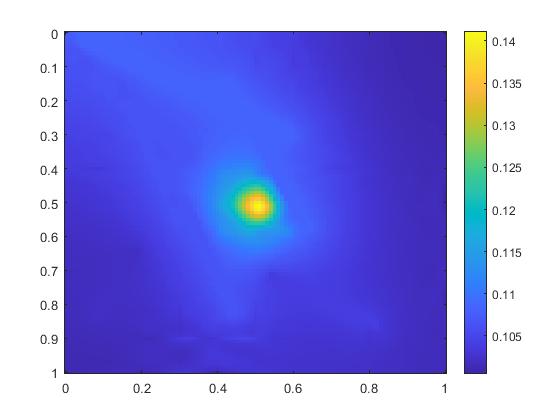}
}
\subfigure{
\includegraphics[width = 5cm]{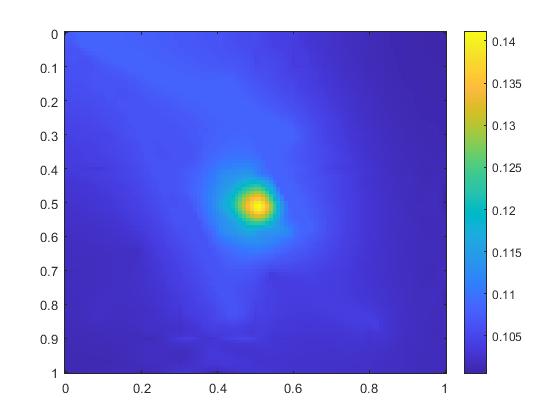}
}
\caption{Left: Reference solution at $t=T$. 
Middle: Implicit CEM solution (with additional basis) at $t=T$.
Right: Partially explicit solution at $t=T$.}
\label{NRfig14}
\end{figure}

\begin{figure}[H]
\centering
\subfigure{
\includegraphics[width = 6cm]{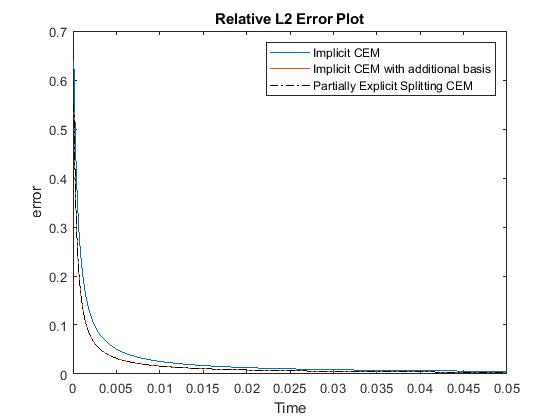}
}
\subfigure{
\includegraphics[width = 6cm]{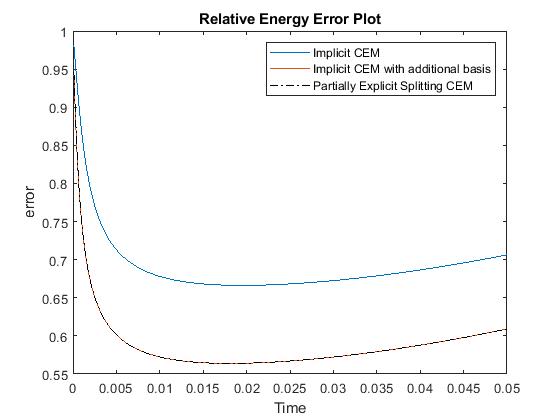}
}
\caption{Left: Relative $L^2$ error. 
Right: Relative energy error.}
\label{NRfig15}
\end{figure}
 
In this case, \[g = -(1+\cos(a_1 \cdot u) + g_0),\; a_1(x,y) = 2\cos(20\pi x)\cos(20\pi y).\]
The permeability field $\kappa$, the source term $g_0$ and the funciotn $a_1$ are presented in Figure~\ref{NRfig16}.
The reference solution, 
implicit CEM solution (with additional basis) and partially explicit solution at the final time are shown in Figure~\ref{NRfig17}. 
We present the relative $L^2$ error plot 
and the relative energy error plot in Figure~\ref{NRfig18}. 
We observe that the $L^2$ and energy error curves for implicit CEM (with additional basis) and partially explicit scheme coincide.  
\begin{figure}[H]
\centering
\subfigure{
\includegraphics[width = 5cm]{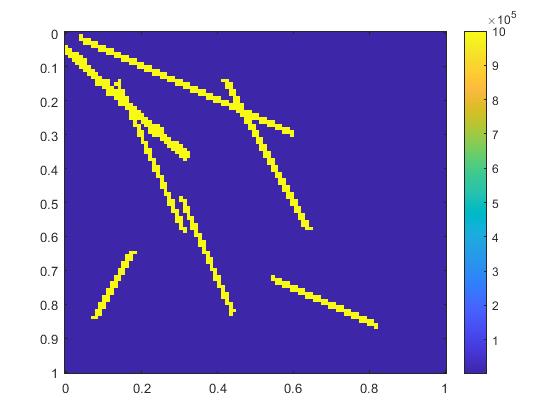}
}
\subfigure{
\includegraphics[width = 5cm]{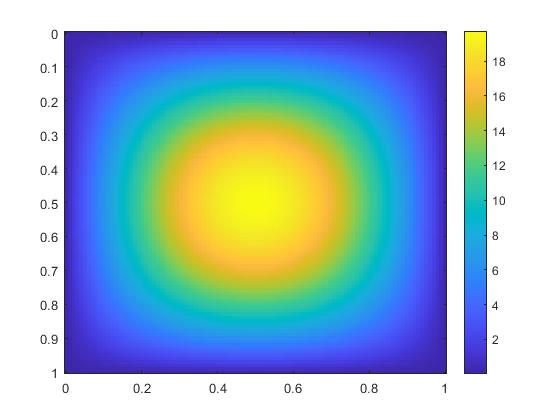}
}
\subfigure{
\includegraphics[width = 5cm]{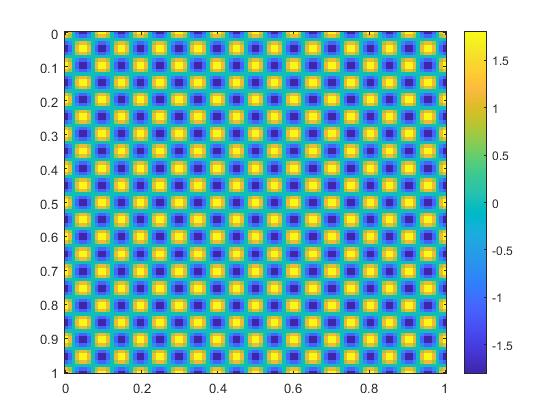}
}
\caption{Left: $\kappa$. Middle: $g_0$. Right: $a_1$.}
\label{NRfig16}
\end{figure}

\begin{figure}[H]
\centering
\subfigure{
\includegraphics[width = 5cm]{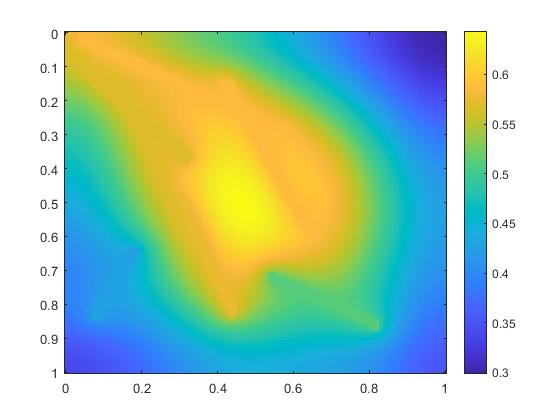}
}
\subfigure{
\includegraphics[width = 5cm]{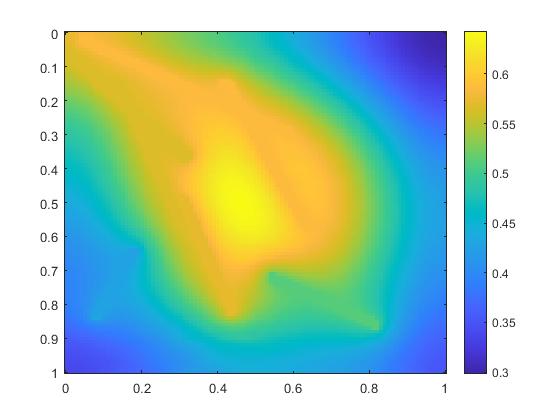}
}
\subfigure{
\includegraphics[width = 5cm]{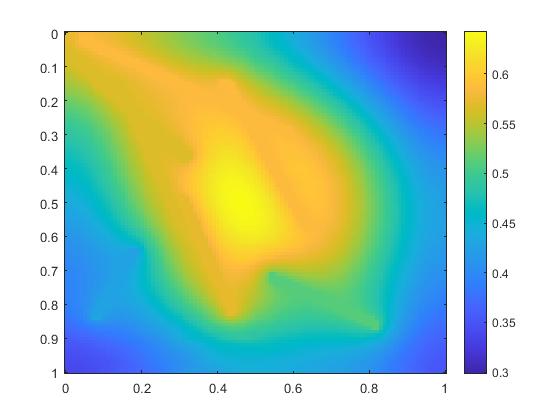}
}
\caption{Left: Reference solution at $t=T$. 
Middle: Implicit CEM solution (with additional basis) at $t=T$.
Right: Partially explicit solution at $t=T$.}
\label{NRfig17}
\end{figure}

\begin{figure}[H]
\centering
\subfigure{
\includegraphics[width = 6cm]{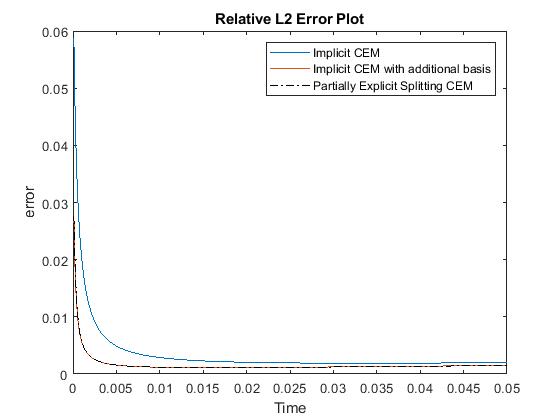}
}
\subfigure{
\includegraphics[width = 6cm]{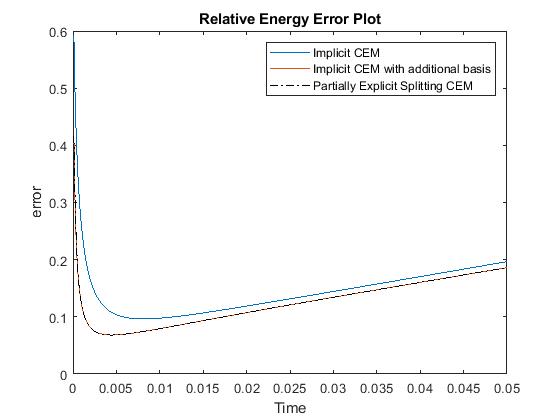}
}
\caption{Left: Relative $L^2$ error. 
Right: Relative energy error.}
\label{NRfig18}
\end{figure}

\subsection{Nonlinear $f(u)$}
We want to explore the case in which the diffusion operator is nonlinear. So in Equation $(\ref{eq1})$, we set 
\[ f(u) = - \nabla \cdot (\kappa \alpha(u)\nabla u).\]
Equation $(\ref{eq1})$ becomes 
\begin{align}
 u_t - \nabla \cdot ( \kappa \alpha(u) \nabla u ) + g(u) = 0. \label{eq72}
\end{align}
Let $u_h$ be the fine mesh solution for Equation $(\ref{eq72})$. We use the Picard-Newton iteration to solve the implicit equation.
\[ \Big( \frac{u^{n+1,m+1}_h-u^n_h}{\Delta t},v \Big) + \int_{\Omega}\kappa \alpha(u^{n+1,m}_h)
\nabla u^{n+1,m+1}_h \cdot \nabla v + (g(u^{n+1,m+1}_h),v) = 0 \quad \forall v\in V_h,  \]
where $m$ is the Picard-Newton step number and $(\cdot,\cdot)$ is the $L^2$ inner product.
In finite element methods, let $\{\varphi_i\}_i$ be fine mesh basis functions.
We have $u^{n+1,m+1}_h= \sum\limits_i U^{n+1,m+1}_{h,i} \varphi_i$, $u^{n+1,m}_h= \sum\limits_i U^{n+1,m}_{h,i} \varphi_i$ and $u^n_h = \sum\limits_i U^n_{h,i} \varphi_i$.
Let $M$ be the mass matrix. Let $U^{n+1,m+1}_h=(U^{n+1,m+1}_{h,i})$, $U^{n+1,m}_h = (U^{n+1,m}_{h,i})$ and $U^n_h = (U^n_{h,i})$.
We define \begin{align*}
Q(U^{n+1,m}_h) = MU^{n+1,m}_h + \Delta t\cdot \mathcal{A} U^{n+1,m}_h - MU^n_h + \Delta t \cdot \mathcal{G},
\end{align*}
where $\mathcal{A} = (\mathcal{A}_{ij})$
\[ \mathcal{A}_{ij} = \int_{\Omega} \kappa \alpha(u^{n+1,m}_h) \nabla \varphi_j \nabla \varphi_i \]
and $\mathcal{G}=(\mathcal{G}_i)$
\[ \mathcal{G}_i = (g(u^{n+1,m}_h) , \varphi_i).\]
Then \[(JQ)(U^{n+1,m}_h) = M +  \Delta t\mathcal{A} + \Delta t \cdot (J\mathcal{G}), \]
where $J\mathcal{G} = ((J\mathcal{G})_{ij})$
\[(J\mathcal{G})_{ij}= \frac{\partial (g(u^{n+1,m}_h),\varphi_i)}{\partial U^{n+1,m}_{h,j}}.\]
Thus, \[U^{n+1,m+1}_h = U^{n+1,m}_h-(JQ)^{-1}(U^{n+1,m}_h) Q(U^{n+1,m}_h).  \]
Newton's method for coarse mesh is similar.
For the partially explicit scheme, we use the following Picard-Newton iteration which can be solved similarly using the method introduced above.
Note that we change the reaction term to be partially explicit.
\begin{align*}
(\frac{u^{n+1,m+1}_{H,1}-u^n_{H,1}}{\Delta t} + \frac{u^n_{H,2}-u^{n-1}_{H,2}}{\Delta t} ,v_1)
+ \int_{\Omega}\kappa \alpha(u^{n+1,m}_{H,1} +u^n_{H,2}) \nabla (u^{n+1,m+1}_{H,1}+u^n_{H,2} ) \cdot \nabla v_1 \\
= ( -g(u^{n+1,m+1}_{H,1} + u^n_{H,2}) , v_1 ) \quad \forall v_1 \in V_{H,1}, \\
( \frac{u^{n+1}_{H,2}-u^n_{H,2}}{\Delta t} + \frac{u^n_{H,1}-u^{n-1}_{H,1}}{\Delta t} , v_2 ) + \int_{\Omega} \kappa \alpha(u^{n+1}_{H,1}+u^n_{H,2}) 
\nabla(u^{n+1}_{H,1} +u^n_{H,2})  \cdot \nabla v_2  \\
=(-g(u^{n+1}_{H,1} + u^n_{H,2}) , v_2 ) \quad \forall v_2 \in V_{H,2}.
\end{align*}
In the following three examples, we use $g(u_{H,1}^{n+1} + u_{H,2}^n) $ in the partially explicit algorithm. One can prove its stability as in Lemma~\ref{lem:lem1} and Lemma~\ref{lem:lem2}. The proof is similar and we omit it here.\\
In the following three cases, \[g(u) = -(10u(u^2-1) + g_0).\]
In the first example, we consider $\alpha(u) = 1+u^2$ and time step $\Delta t = \frac{T}{500}=10^{-4}$. The permeability field $\kappa$ and the source term $g_0$ are presented in Figure~\ref{NRfig19}.
The reference solution, implicit CEM solution (with additional basis) and partially explicit solution at the final time are shown in Figure~\ref{NRfig20}.
We present the relative $L^2$ error plot and relative energy error plot in Figure~\ref{NRfig21}.
We can find that the curves for implicit CEM solution (with additional basis) and partially explicit solution coincide, which implies that our new partially explicit scheme is also effective and has similar accuracy as the implicit CEM (with additional basis) scheme. 
\begin{figure}[H]
\centering
\subfigure{
\includegraphics[width = 6cm]{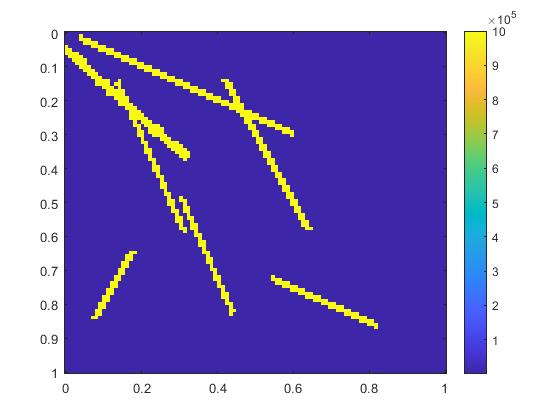}
}
\subfigure{
\includegraphics[width = 6cm]{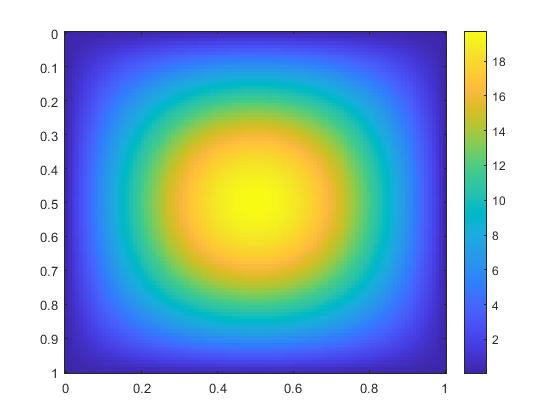}
}
\caption{Left: $\kappa$. Right: $g_0$.}
\label{NRfig19}
\end{figure}

\begin{figure}[H]
\centering
\subfigure{
\includegraphics[width = 5cm]{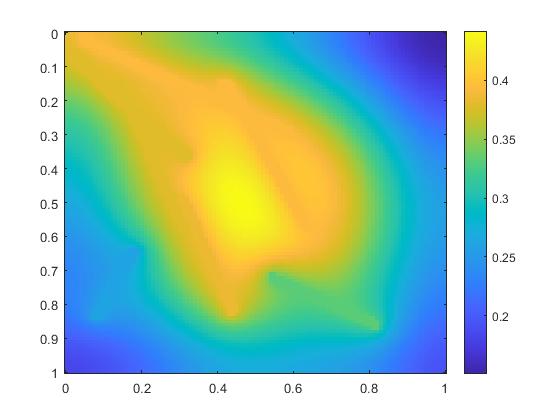}
}
\subfigure{
\includegraphics[width = 5cm]{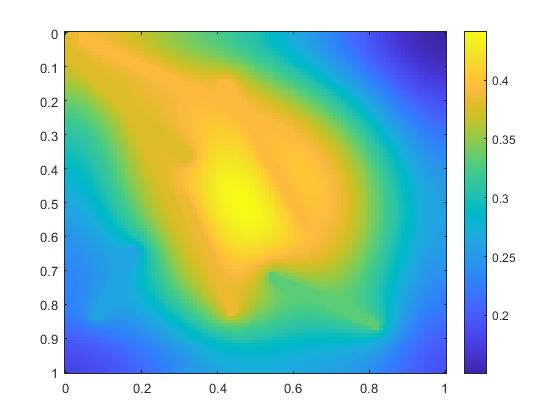}
}
\subfigure{
\includegraphics[width = 5cm]{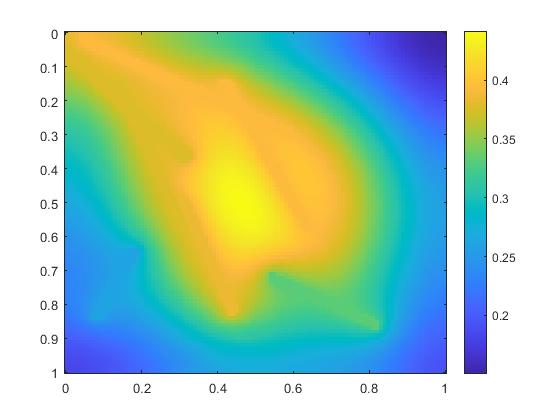}
}
\caption{Left: Reference solution at $t=T$. 
Middle: Implicit CEM solution (with additional basis) at $t=T$.
Right: Partially explicit solution at $t=T$.}
\label{NRfig20}
\end{figure}

\begin{figure}[H]
\centering
\subfigure{
\includegraphics[width = 6cm]{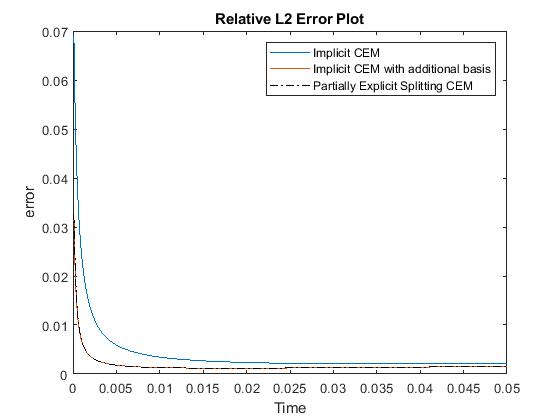}
}
\subfigure{
\includegraphics[width = 6cm]{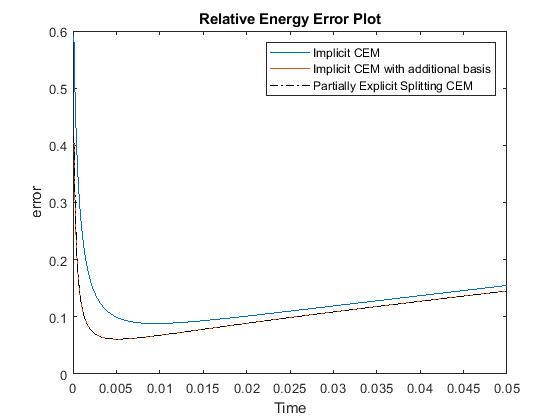}
}
\caption{Left: Relative $L^2$ error. 
Right: Relative energy error.}
\label{NRfig21}
\end{figure}

In this second example, we consider $\alpha(u) = 2 + \cos(u) $ and time step $\Delta t = \frac{0.05}{1500}$.
Figure~\ref{NRfig22} shows the permeability field $\kappa$ and the source term $g_0$. 
The reference solution, implicit CEM solution (with additional basis) and partially explicit solution at $t=T$ are presented in Figure~\ref{NRfig23}.
The relative $L^2$ error plot and relative energy error plot are shown in Figure~\ref{NRfig24}. 
The partially explicit scheme also works in this case and has similar accuracy as the implicit CEM (with additional basis) scheme.
\begin{figure}[H]
\centering
\subfigure{
\includegraphics[width = 6cm]{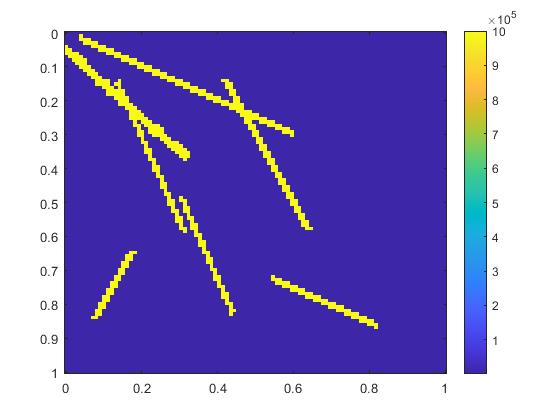}
}
\subfigure{
\includegraphics[width = 6cm]{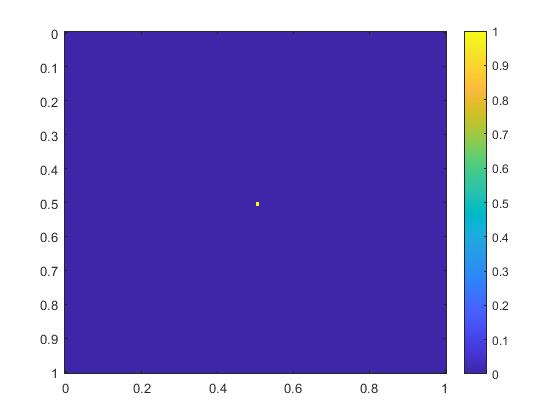}
}
\caption{Left: $\kappa$. Right: $g_0$.}
\label{NRfig22}
\end{figure}

\begin{figure}[H]
\centering
\subfigure{
\includegraphics[width = 5cm]{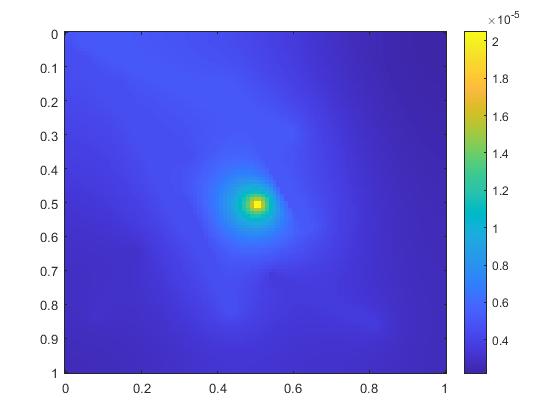}
}
\subfigure{
\includegraphics[width = 5cm]{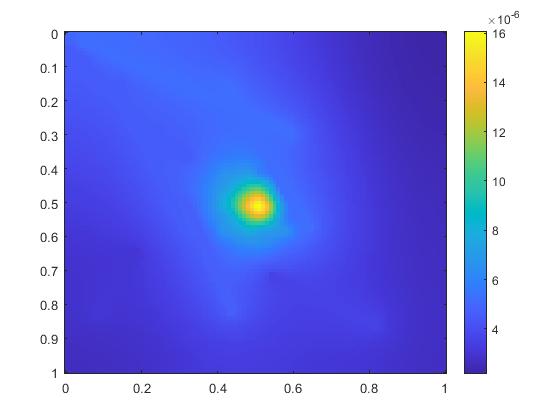}
}
\subfigure{
\includegraphics[width = 5cm]{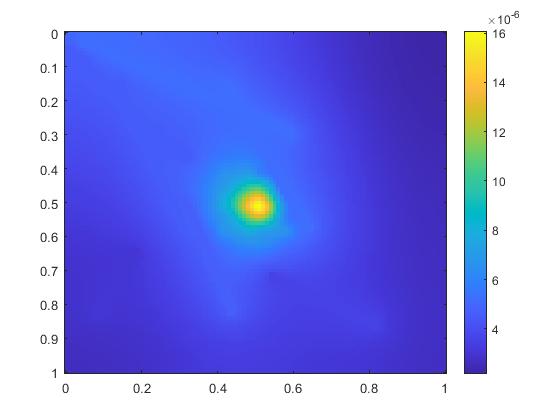}
}
\caption{Left: Reference solution at $t=T$. 
Middle: Implicit CEM solution (with additional basis) at $t=T$.
Right: Partially explicit solution at $t=T$.}
\label{NRfig23}
\end{figure}

\begin{figure}[H]
\centering
\subfigure{
\includegraphics[width = 6cm]{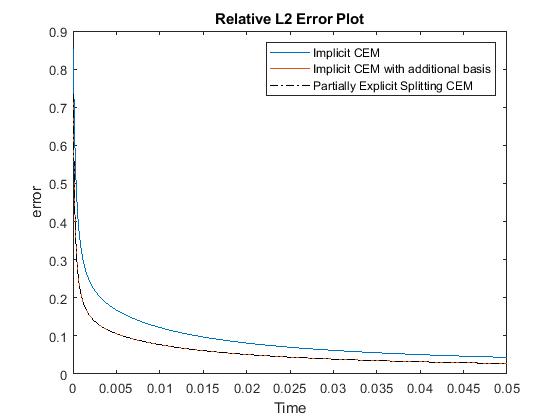}
}
\subfigure{
\includegraphics[width = 6cm]{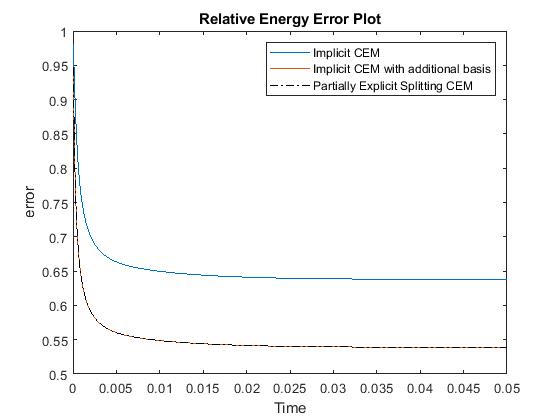}
}
\caption{Left: Relative $L^2$ error. 
Right: Relative energy error.}
\label{NRfig24}
\end{figure}

In this case, we have $\alpha(u) = 2 + \cos(u) $ and time step $\Delta t = \frac{0.05}{1500}$. The permeability field $\kappa$ and the source term $g_0$ are presented in Figure~\ref{NRfig25}. 
We show the reference solution, implicit CEM solution (with additional basis) and partially explicit solution at $t=T$ in Figure~\ref{NRfig26}.
The relative $L^2$ error plot and the relative energy error plot are presented in Figure~\ref{NRfig27}.
From Figure~\ref{NRfig27}, we see that the curves for implicit CEM (with additional basis) and partially explicit scheme coincide, which implies that they have similar accuracy.
\begin{figure}[H]
\centering
\subfigure{
\includegraphics[width = 6cm]{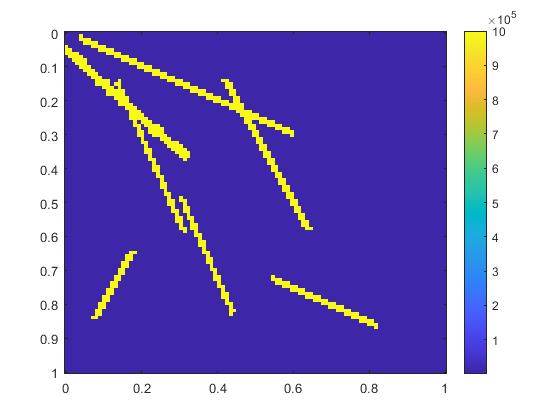}
}
\subfigure{
\includegraphics[width = 6cm]{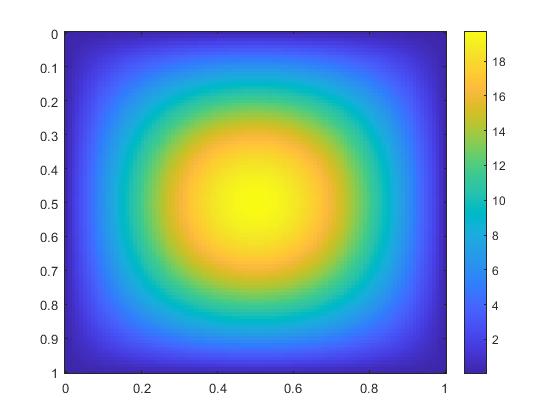}
}
\caption{Left: $\kappa$. Right: $g_0$.}
\label{NRfig25}
\end{figure}

\begin{figure}[H]
\centering
\subfigure{
\includegraphics[width = 5cm]{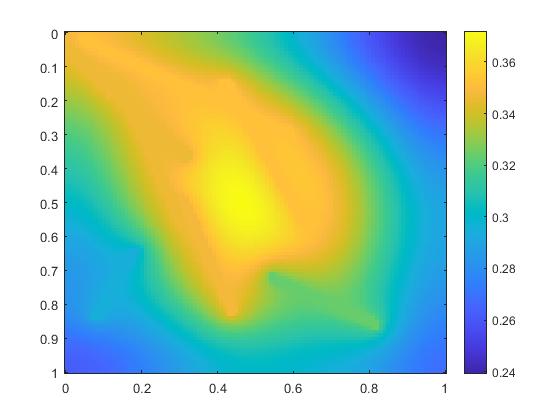}
}
\subfigure{
\includegraphics[width = 5cm]{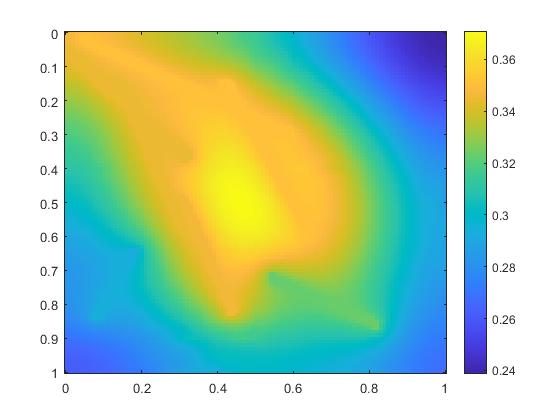}
}
\subfigure{
\includegraphics[width = 5cm]{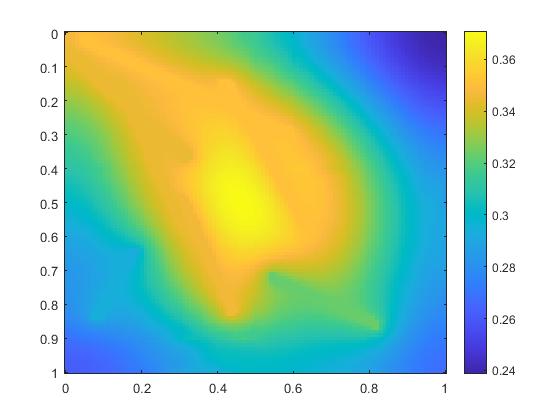}
}
\caption{Left: Reference solution at $t=T$. 
Middle: Implicit CEM solution (with additional basis) at $t=T$.
Right: Partially explicit solution at $t=T$.}
\label{NRfig26}
\end{figure}

\begin{figure}[H]
\centering
\subfigure{
\includegraphics[width = 6cm]{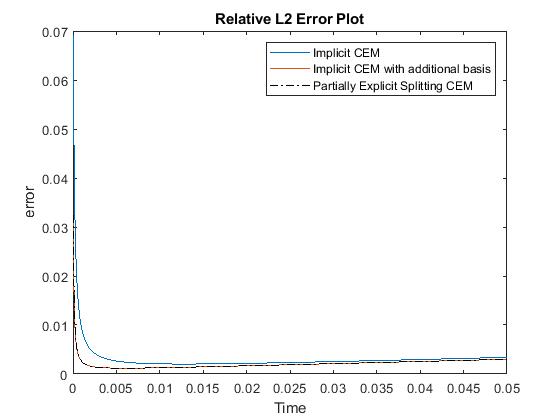}
}
\subfigure{
\includegraphics[width = 6cm]{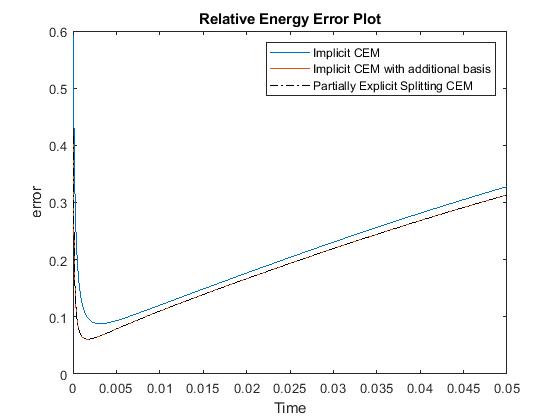}
}
\caption{Left: Relative $L^2$ error. 
Right: Relative energy error.}
\label{NRfig27}
\end{figure}

\section{Conclusions}

In this work, we design and analyze contrast-independent time
discretization for nonlinear problems. 
The work continues our earlier works on linear problems,
where we propose temporal splitting and associated spatial decomposition
that guarantees a stability. 
We introduce two spatial spaces, the first account for spatial features 
related to fast time scales and the second for spatial features related 
to ``slow'' time scales. We propose time splitting, where the first equation solves for fast components implicitly and the second equation solves for slow components explicitly.
We introduce a condition for multiscale spaces
that guarantees stability of the proposed splitting algorithm.
 Our proposed method is still implicit via mass matrix; however, 
it is explicit in terms of stiffness matrix for the slow component.
 We present numerical results, which show that the proposed methods provide very similar results as fully implicit methods using explicit methods with the time stepping that is independent of the contrast.

\section*{Acknowledgments}

The research of Eric Chung is partially supported by the Hong Kong RGC General Research Fund (Project numbers 14304719 and 14302018) and the CUHK Faculty of Science Direct Grant 2020-21.

\bibliographystyle{abbrv}
\bibliography{references,references4,references1,references2,references3,decSol}

\end{document}